\documentclass[12pt,reqno
 ]{amsart}

\RequirePackage{mathtools}
\RequirePackage{amsopn}
\RequirePackage{amsfonts}
\RequirePackage{paralist}
\RequirePackage{amssymb}
\RequirePackage{amsthm}
\RequirePackage{mathrsfs}
\usepackage{enumitem}
\usepackage[alphabetic]{amsrefs}
\renewcommand\MR[1]{\relax} 
\usepackage{tikz}
\usetikzlibrary{cd,decorations.pathmorphing,decorations.markings}
\mathtoolsset{showonlyrefs}
\newtheorem{thm}{Theorem}[section] 
    \newtheorem{thmx}{Theorem}
\numberwithin{equation}{section}

\newtheorem{cor}[thm]{Corollary}
\newtheorem{corollary}[thm]{Corollary}
\newtheorem{lemma}[thm]{Lemma}
\newtheorem{prop}[thm]{Proposition}

\theoremstyle{definition}
\newtheorem{definition}[thm]{Definition}
\newtheorem{claim}{Claim}
\newtheorem*{claim*}{Claim}

\theoremstyle{remark}
\newtheorem{remark}[thm]{Remark}
\newtheorem{example}[thm]{Example}
\newtheorem{question}[thm]{Question}

\def\mathcs{C^{*}}
\newcommand{\cs}{\ensuremath{\mathcs}}
\def\Cst{\cs}

\DeclareMathSymbol{\rtimes}{\mathbin}{AMSb}{"6F}

\newcommand\C{\mathbf{C}}
\newcommand\T{\mathbf{T}}
\newcommand\Z{\mathbf{Z}}
\newcommand\Q{\mathbf{Q}}
\newcommand\N{\mathbf{N}}

\newcommand\set[1]{\{\,#1\,\}}
\newcommand\sset[1]{\{#1\}}
\def\restr#1{|_{{#1}}}
\makeatletter
\def\labelenumi{\textnormal{(\@alph\c@enumi)}}
\def\theenumi{\@alph \c@enumi}
\def\labelenumii{\textnormal{(\@roman\c@enumii)}}
\def\theenumii{\@roman \c@enumii}
\newcount\charno
\def\alphapart#1{\charno=96
\advance\charno by#1\char\charno}

\makeatother

 \def\<{\langle}
 \def\>{\rangle}
\let\ipscriptstyle=\scriptscriptstyle
\def\lipsqueeze{{\mskip -3.0mu}}
\def\ripsqueeze{{\mskip -3.0mu}}
\def\ipcomma{\nobreak\mathrel{,}\nobreak}
\newbox\ipstrutbox
\setbox\ipstrutbox=\hbox{\vrule height8.5pt
width 0pt}
\def\ipstrut{\copy\ipstrutbox}
\def\lip#1<#2,#3>{\mathopen{\relax_{\ipstrut\ipscriptstyle{
#1}}\lipsqueeze
\langle} #2\ipcomma #3 \rangle}
\def\blip#1<#2,#3>{\mathopen{\relax_{\ipstrut
\ipscriptstyle{ #1}}\lipsqueeze\bigl\langle} #2\ipcomma #3 \bigr\rangle}
\def\rip#1<#2,#3>{\langle #2\ipcomma #3
\rangle_{\ripsqueeze\ipstrut\ipscriptstyle{#1}}}
\def\brip#1<#2,#3>{\bigl\langle #2\ipcomma #3
\bigr\rangle_{\ripsqueeze\ipstrut\ipscriptstyle{#1}}}
\def\angsqueeze{\mskip -6mu}
\def\smangsqueeze{\mskip -3.7mu}
\def\trip#1<#2,#3>{\langle\smangsqueeze\langle #2\ipcomma #3
\rangle\smangsqueeze\rangle_{\ripsqueeze\ipstrut\ipscriptstyle{#1}}}
\def\btrip#1<#2,#3>{\bigl\langle\angsqueeze\bigl\langle #2\ipcomma
#3
\bigr\rangle
\angsqueeze\bigr\rangle_{\ripsqueeze\ipstrut\ipscriptstyle{#1}}}
\def\tlip#1<#2,#3>{\mathopen{\relax_{\ipstrut\ipscriptstyle{
#1}}\lipsqueeze \langle\smangsqueeze\langle} #2\ipcomma #3
\rangle\smangsqueeze\rangle}
\def\btlip#1<#2,#3>{\mathopen{\relax_{\ipstrut\ipscriptstyle{
#1}}\lipsqueeze
\bigl\langle\angsqueeze\bigl\langle} #2\ipcomma #3
\bigr\rangle\angsqueeze\bigr\rangle}

\def\ip(#1|#2){(#1\mid #2)}
\def\bip(#1|#2){\bigl(#1 \mid #2\bigr)}
\def\Bip(#1|#2){\Bigl( #1 \bigm| #2 \Bigr)}
\expandafter\ifx\csname BibSpec\endcsname\relax\else
\BibSpec{collection.article}{%
    +{}  {\PrintAuthors}                {author}
    +{,} { \textit}                     {title}
    +{.} { }                            {part}
    +{:} { \textit}                     {subtitle}
    +{,} { \PrintContributions}         {contribution}
    +{,} { \PrintConference}            {conference}
    +{}  {\PrintBook}                   {book}
    +{,} { }                            {booktitle}
    +{,} { }                            {series}
    +{,} { \voltext}                    {volume}
    +{,} { }                            {publisher}
    +{,} { }                            {organization}
    +{,} { }                            {address}
    +{,} { \PrintDateB}                 {date}
    +{,} { pp.~}                        {pages}
    +{,} { }                            {status}
    +{,} { \PrintDOI}                   {doi}
    +{,} { available at \eprint}        {eprint}
    +{}  { \parenthesize}               {language}
    +{}  { \PrintTranslation}           {translation}
    +{;} { \PrintReprint}               {reprint}
    +{.} { }                            {note}
    +{.} {}                             {transition}
}
\BibSpec{article}{%
    +{}  {\PrintAuthors}                {author}
    +{,} { \textit}                     {title}
    +{.} { }                            {part}
    +{:} { \textit}                     {subtitle}
    +{,} { \PrintContributions}         {contribution}
    +{.} { \PrintPartials}              {partial}
    +{,} { }                            {journal}
    +{}  { \textbf}                     {volume}
    +{}  { \PrintDatePV}                {date}
    +{,} { \eprintpages}                {pages}
    +{,} { }                            {status}
    +{,} { \PrintDOI}                   {doi}
    +{,} { available at \eprint}        {eprint}
    +{}  { \parenthesize}               {language}
    +{}  { \PrintTranslation}           {translation}
    +{;} { \PrintReprint}               {reprint}
    +{.} { }                            {note}
    +{.} {}                             {transition}
}
\BibSpec{book}{%
    +{}  {\PrintPrimary}                {transition}
    +{,} { \textit}                     {title}
    +{.} { }                            {part}
    +{:} { \textit}                     {subtitle}
    +{,} { \PrintEdition}               {edition}
    +{}  { \PrintEditorsB}              {editor}
    +{,} { \PrintTranslatorsC}          {translator}
    +{,} { \PrintContributions}         {contribution}
    +{,} { }                            {series}
    +{,} { \voltext}                    {volume}
    +{,} { }                            {publisher}
    +{,} { }                            {organization}
    +{,} { }                            {address}
    +{,} { pp.~}                        {pages}
    +{,} { \PrintDateB}                 {date}
    +{,} { }                            {status}
    +{}  { \parenthesize}               {language}
    +{}  { \PrintTranslation}           {translation}
    +{;} { \PrintReprint}               {reprint}
    +{.} { }                            {note}
    +{.} {}                             {transition}
}
\fi
\evensidemargin=\oddsidemargin
\makeatletter
    \newcommand{\repeatable}[2]{%
        \label{#1}\global\@namedef{repeatable@#1}{#2}#2%
    }

    \newcommand{\txtrepeat}[1]{%
        \@ifundefined{repeatable@#1}{NOT FOUND}{\@nameuse{repeatable@#1}}%
    }

    \newcommand{\eqrepeat}[1]{%
        \@ifundefined{repeatable@#1}{NOT FOUND}{\begin{align*}\@nameuse{repeatable@#1}\tag{\ref{#1}}\end{align*}}%
    }
    
    \newcommand{\eqrepeatnn}[1]{%
        \@ifundefined{repeatable@#1}{NOT FOUND}{\begin{align*}\@nameuse{repeatable@#1}\end{align*}}%
    }
\makeatother

\newcommand\z{^{(0)}}
\newcommand\go{G\z}
\newcommand\comp{^{(2)}}
\newcommand\inv{^{-1}}

\newcommand{\Iso}[1]{{#1'}}
\newcommand{\Int}[2][]{\operatorname{int}_{#1}\left(#2\right)}
\DeclareMathOperator{\supp}{supp}
\newcommand\suppo{\supp^{\circ}}
\newcommand\E{\mathcal{E}}
\newcommand\s{\mathfrak{s}}
\newcommand\norm[1]{\|#1\|}
\newcommand\abs[1]{|#1|}

\newcommand{\etale}{{\'e}tale}
\newcommand{\LCH}{locally compact, Hausdorff}

\allowdisplaybreaks
\begin{document}

\begin{abstract}
 Well-known
 work of Renault shows that if $\E$ is a twist over a 
 second countable, effective, \etale\ groupoid $G$, then there is
  a naturally associated Cartan subalgebra of the reduced twisted
  groupoid \cs-algebra $\cs_{r}(G;\E)$, and that every Cartan
  subalgebra of 
  a
  separable
  \cs-algebra arises in this way.  However
 twisted $\cs$-algebras of non-effective groupoids $G$ can also possess Cartan subalgebras: In  \cite{DGNRW:Cartan}, sufficient conditions on a subgroupoid $S$ of $G$ were found that ensure that $S$ gives rise to a Cartan subalgebra in the cocycle-twisted $\cs$-algebra of $G$. In this paper, 
  we extend these results to general twists $\E$, and we refine the conditions on the subgroupoid 
  for  $\cs_{r}(S;\E_S)$ to be a Cartan subalgebra of $\cs_{r}(G;\E)$. 
\end{abstract}

\thanks{Anna Duwenig was supported by a RITA Investigator 
grant (IV017) of the University of Wollongong, by Methusalem grant
  METH/21/03---long term structural funding of the Flemish Government,
  and by an FWO Senior Postdoctoral Fellowship (1206124N). She would 
  like to thank Diego Mart\'inez and her WOA1 and WOA3 groups for helpful discussions, and she is particularly grateful to  Ying-Fen Lin
  and Adam Fuller whose questions helped to find gaps in earlier versions%
  }

\title{Non-traditional Cartan subalgebras in twisted groupoid \cs-algebras}

\author[A. Duwenig]{Anna Duwenig}
\address{KU Leuven, Department of Mathematics, Leuven (Belgium)}
\email{anna.duwenig@kuleuven.be}

\author[Williams]{Dana P. Williams}
\address{Department of Mathematics\\ Dartmouth College \\ Hanover, NH
  03755-3551 USA}
\email{dana.williams@Dartmouth.edu}

\author[J. Zimmerman]{Joel Zimmerman}
\address{School
  of Mathematics and Applied Statistics, University of Wollongong,
  Wollongong, NSW 2522, Australia}
\email{joelz@uow.edu.au}

\maketitle


\section{Introduction}

\label{sec:introduction}

Cartan subalgebras have been an active area of research in operator
algebras since the work of Feldman and Moore \cite{FeldmanMoore2} who
realized von Neumann algebras with Cartan subalgebras as the von
Neumann algebras of Borel equivalence relations twisted by Borel
$2$-cocycles.  Generalizing the Feldman--Moore Theorem to the setting
of \cs-algebras was one of the impetuses for studying the \cs-algebras
of locally compact groupoids.  Unfortunately, the basic theory only
accommodates continuous $2$-cocycles since the groupoid extension
determined by a Borel $2$-cocycle need not have a compatible locally
compact topology as is the case for locally compact groups.  A very
useful solution to this dilemma was provided by Kumjian in
\cite{Kum:Diags} where he suggested dispensing with explicit mention
of the $2$-cocycle on a groupoid $G$ and to instead work with a
groupoid extension of~$G$ which is a natural groupoid analogue of
central group extensions of the circle $\T$.  He called these
extensions \emph{twists} over $G$.  The term $\T$-groupoid is also
used.  Then given a twist $\E$ over $G$, Kumjian showed how to
construct a \cs-algebra $\cs(G;\E)$, and its reduced counterpart
$\cs_{r}(G;\E)$, which are analogues of the classical group
\cs-algebras twisted by a
$2$-cocycle.

There has been a great deal of work invested in the study of these
twisted groupoid \cs-algebras.  In particular, in
\cite{ren:irms08} Renault established a correspondence between
separable \cs-algebras with Cartan subalgebras and twists over 
\LCH,
second countable, effective, \etale\ groupoids.\footnote{It was later shown that the same result holds in the non-separable / non-second-countable case. See for example  \cite{Raad:2022:Renault}*{Theorem 1.2} or \cite{KM:2020:NCCartans}*{Corollary 7.6}.}  But there are twists over
\emph{non-effective} groupoids whose twisted \cs-algebras possess
Cartan subalgebras. Such examples can be constructed using Theorem~3.1
of \cite{DGNRW:Cartan} for twists 
that are
determined by a continuous $2$-cocycle.

In this paper, we generalize \cite{DGNRW:Cartan}*{Theorem~3.1} to all twists, and we refine the sufficient conditions
 conditions on $S$ for $\Cst_{r}(S;\E_S)$ to be a Cartan subalgebra. %
Our main result (Theorem~\ref{thm:May17}) implies the following:
\begin{thmx}\label{thmA}
   Suppose 
  $\E$ is a twist 
  over a \LCH, \etale\ groupoid~$G$, and 
  $S$
  is an open subgroupoid of $G$. 
  Then the following statements  are equivalent.
  \begin{enumerate}[label=\textup{(\roman*)}]
    
    \item $i(\Cst_{r}(S;\E_{S}))$ is a Cartan subalgebra of $\Cst_{r}(G;\E)$ whose normalizer contains every $h\in C_{c}(G;\E)$ with  $\pi(\suppo(h))$ a bisection.
    
    \item  $S$ is maximal among open subgroupoids of $\Int[G]{\Iso{G}}$ for which
      $\E_{S}$ is abelian, $S$ is closed and normal in $G$, and the set
      \begin{equation}
        \bigl\{e\in\Iso{\E} :  1 < |\{\sigma\inv e\sigma:\sigma \in \E_{S}\}|= |\{s\inv \pi(e)s:s \in S\}|<\infty\bigr\}
        \end{equation}
        has empty interior in $\E$.
  \end{enumerate}
\end{thmx}

 A notable ingredient to proving 
 our results
 is Proposition~\ref{prop:max_new}, in which we show which clopen subgroupoids give rise to maximal abelian subalgebras.
In
 Section~\ref{sec:applications}, we prove that we recover multiple known results in the literature and we apply our main theorem to some explicit examples.
 We conclude  with open questions in our final Section~\ref{sec:Qs}.

\section*{Notation and Conventions}

All our groupoids are assumed to be locally compact and Hausdorff.
We say that 
a groupoid
$G$ is \emph{abelian} if it is a bundle of abelian
groups. For a unit $u\in \go$, we write
$G^{u}\coloneq r\inv (\sset{u})$ and
$G_{u}\coloneq s\inv (\sset{u})$. The \emph{isotropy group} at $u$ is
denoted $G(u)\coloneq G^{u}\cap G_{u}$, and the \emph{isotropy
  bundle} by $\Iso{G}=\bigcup_{u\in \go}G(u)$.  A groupoid is
\emph{\etale} if the range and source maps are local homeomorphisms.
This implies $\go$ is open in~$G$.  If in addition, the interior of
$\Iso{G}$ is reduced to $\go$, then $G$ is called \emph{effective}.
A
subgroupoid $S$ in $\Iso{G}$ is called \emph{normal} if $Sg=g S$
for all $g\in G$ so that $G$ acts on~$S$ by conjugation. 

Given a topological space $X$, we will denote the interior of a
subspace $Y$ by $\Int[X]{Y}$; if (or as soon as) the ambient space is
understood, we will instead write $Y^{\circ}$.  We will denote the
\emph{open support} of a function $f$ by $\suppo(f)$, while $\supp(f)$
will as usual denote its closure.

\section{Twists and their \cs-Algebras}
\label{sec:twists-their-cs}

Groupoid twists---sometimes called $\T$-groupoids---were introduced by
Kumjian in \cite{Kum:Diags} and are now ubiquitous in the subject.
Twists over \etale\ groupoids are discussed in detail in
\cite{Sims:gpds}*{\S11.1}.  In general, recall that a twist $\E$ over
$G$ is given by a central groupoid extension
\begin{equation}
  \label{eq:50}
  \begin{tikzcd}
    \go\times \T\arrow[r,"\iota",hook] &\E \arrow[r,"\pi",two heads]&
    G
  \end{tikzcd}
\end{equation}
where $\iota$ and $\pi$ are continuous groupoid homomorphisms such
that $\iota$ is a homeomorphism onto its range, and such that $\pi$ is
an open surjection inducing a homeomorphism of the unit space of $\E$
with $\go$ with kernel equal to the range of $\iota$.  We further
assume that $\iota(\pi(r(e)),1)=r(e)$ for all $e\in \E$, so that we
can identify the unit space of $\E$ with $\go$.  Furthermore that the
extension be central means that
\begin{equation}
  \label{eq:51}
  \iota\bigl(r(e),z\bigr)e=e\iota\bigl (s(e),z \bigr)\quad\text{for
    all $e\in \E$ and $z\in \T$.}
\end{equation}
Note that $\E$ becomes a principal $\T$-bundle with respect to the
action $z\cdot e=\iota(r(e),z)e$ in such a way that $\pi$ induces a
homeomorphism of the the orbit space $\E/\T$ with $G$.
Conversely, we can think of a twist as a groupoid $\E$ admitting a
free left $\T$-action that is compatible with the groupoid
structure---see \cite{wykwil:jot22}*{Lemma~6.1} for details.

\begin{example}\label{ex:twist given by $2$-cocycle}
  Throughout 
  we will  refer back
  to the special case
  treated in \cite{DGNRW:Cartan} where the twist over $G$ is induced
  by a continuous 
  \emph{normalized $2$-cocycle}
  ${c}\colon G\comp \to \T$, meaning that
  ${c}$ satisfies ${c}(g_{1},g_{2}){c}(g_{1}g_{2}, g_{3})={c}(g_{1},g_{2}g_{3}){c}(g_{2}, g_{3})$ for all $(g_{1},g_{2},g_{3})\in G^{(3)}$ and ${c}(g,s(g))=1={c}(r(g),g)$ for all $g\in G$. The associated twist
  $\E_{{c}}$ is defined as follows.
  As a set, 
  $\E_{{c}}$ is given by $G\times\T$, but the multiplication is
  twisted by the $2$-cocycle:
  $(g_{1},z_{1})\cdot(g_{2},z_{2}) =
  (g_{1}g_{2},{c}(g_{1},g_{2})z_{1}z_{2})$.
  In this case, we choose $\iota$ to be the inclusion map and $\pi$ to be the projection
  onto the first component.
\end{example}

Given a twist $\E$ over $G$ as in \eqref{eq:50}, we let
\begin{equation}
  \label{eq:1}
  C_{0}(G;\E)=\set{f\in C_{0}(\E):\text{$f(z\cdot e)=zf(e)$ for all
      $z\in\T$ and $e\in \E$}},
\end{equation}
and let $C_{c}(G;\E)=\set{f\in C_{0}(G;\E):f\in C_{c}(\E)}$.

Note that if $f_{1},f_{2}\in C_{0}(G;\E)$, then for any $e'\in \E$,
$\pi(e) \mapsto f_{1}(e)f_{2}(e\inv e')$ is a well-defined function on
$G^{r(e')}$.  If $G$ is \etale, equipped with counting measures as a
Haar system, then we therefore get a well-defined convolution product on
$C_{c}(G;\E)$ by
\begin{equation}
  \label{eq:2}
  f_{1}*f_{2}(e') = \sum_{\set{\pi(e)\in G:r(e)=r(e')}} f_{1}(e)f_{2}(e\inv e').
\end{equation}
Combined with the involution $f^{*}(e)=\overline{f(e\inv )}$, this
makes $C_{c}(G;\E)$ into a $*$-algebra.

\begin{remark}[Pedantry]
  \label{rem-pedantry} The notation in \eqref{eq:2} is meant to ensure
  that the sum is interpreted as being over the countable set
  $G^{r(e')}$ and not over the uncountable set~$\E^{r(e')}$.  A more
  pedantic approach would be to introduce any set-theoretic section
  $\s\colon G\to\E$ for $\pi$ and write
  \begin{equation}
    \label{eq:3}
    f_{1}*f_{2}(e') =\sum_{g\in G^{r(e')}} f_{1}(\s(g))f_{2}(\s(g)\inv e'),
  \end{equation}
  and to observe that the left-hand side of \eqref{eq:3} does not
  depend on the choice of section.  In the spirit of Winston
  Churchill's reaction to allegedly misplaced prepositions, we have
  chosen to write \eqref{eq:2} in place of \eqref{eq:3} to keep our
  formulas less daunting.  It also helps us connect to more general
  treatments, where $G$ is not necessarily \etale, such as
  \cite{MW:1992:CtsTrace} where the right-hand side of \eqref{eq:2} is
  simply the integral of the function $\pi(e)\mapsto f_{1}(e)f_{2}(e\inv e')$
  with respect to counting measure on $G^{r(e')}$.
\end{remark}

Given a twist $\E$ over an \etale\ groupoid $G$, we can form the
\emph{reduced twisted groupoid \cs-algebra} $\cs_{r}(G;\E)$ as the
completion of $C_{c}(G;\E)$ with respect to the reduced norm~\mbox{$\|\cdot\|_{r}$.}  For the sake of completeness, we review this
construction. As we assume $G$ to be \etale, $G_{u}$ is
discrete for all $u\in\go$.  Then, as in
\cite{MW:1992:CtsTrace}*{\S3}, we let $\ell^{2}(G_{u};\E_{u})$ be the
set of functions $\xi\colon \E_{u}\to \C$ such that
$\xi(z\cdot e)=z\xi(e)$ for all $e\in\E$ and $z\in\T$, and such that
\begin{equation}
  \label{eq:11}
  \sum_{\set{\pi(e)\in G:s(e)=u}}|\xi(e)|^{2}<\infty
\end{equation}
with the same understanding as in Remark~\ref{rem-pedantry}.  Then, as
in \cite{wil:toolkit}*{Ex~3.7.1}, it is not hard to see that
$\ell^{2}(G_{u};\E_{u})$ is a Hilbert space with respect to the inner
product
\begin{equation}
  \label{eq:4}
  \rip u<f_{1},f_{2}>=\sum_{\set{\pi(e)\in G:s(e)=u}} f_{1}(e)\overline {f_{2}(e)}
  .
\end{equation}
Just as in \cite{MW:1992:CtsTrace}*{\S3}, it is not hard to check that
$C_{c}(G_{u}; \E_{u})$ is dense in $\ell^{2}(G_{u};\E_{u})$.

We
define
$\theta_{u}\colon C_{c}(G;\E)\to B\bigl(\ell^{2}(G_{u};\E_{u})\bigr)$
by
\begin{equation}
  \label{eq:5}
  \theta_{u}(f)(\xi)=f*\xi\quad\text{for $\xi\in \ell^{2}(G_{u};\E_{u})$,}
\end{equation}
where the convolution is defined by the same formula as in~\eqref{eq:3}.
Then for $f\in C_{c}(G;\E)$ we define
\begin{equation}
  \label{eq:6}
  \|f\|_{r}=\sup_{u\in\go}\|\theta_{u}(f)\|.
\end{equation}
Since $\|\cdot\|_{r}$ is a norm on $C_{c}(G;\E)$, we can view
$C_{c}(G;\E)$ as a $*$-subalgebra of $\cs_{r}(G;\E)$.
We will (often tacitly) make frequent use of the following well-known result.
\begin{lemma}\label{lem:Urysohn}
    Suppose $\E$ is a twist over a  \LCH\ \etale\ groupoid~$G$. For any $e_0\in \E$, there exists an element $f\in C_{c}(G;\E)$ such that $f(e_0)\neq 0$ and $\pi(\suppo(f))$ is a bisection.
\end{lemma}

\begin{proof}
  By Tietze's extension theorem, we may find $g_0\in C_c(\E)$ such
  that $g_0(z\cdot e_0)=z$; the function
  $g\colon e\mapsto \int_\T z g_0(\bar z\cdot e)\,\mathrm{d}z$ is then in
  $C_c(G;\E)$ and satisfies $g(e_0)=1$.  Since $G$ is \etale, it
  has a basis of open bisections
  \cite{Renault:gpd-approach}*{Proposition~I.2.8}, so by Urysohn's, we
  can find a $q\in C_{c}(G)$ such that $\suppo(q)$ is a bisection of
  $G$ with $\pi(e_0) \in \suppo(q)$; in particular, $f\colon e\mapsto q(\pi(e))g(e)$ is an element of $C_{c}(G;\E)$ such that $f(e_0)\neq 0$ and $\pi(\suppo(f))$ is a bisection.
\end{proof}

\begin{example}
  \label{ex-abelian} Suppose that $\E$ is a twist over an \etale\
  groupoid $G$ and that $\E$ is a bundle of abelian groups.  That is,
  we assume that $\E$ is an abelian twist.  Of course, this implies
  that $G$ is abelian as well.  Then $C_{c}(G;\E)$ is commutative.
  Hence $\cs_{r}(G;\E)$ is commutative.

  Conversely, if $\E$ is not a bundle of abelian groups, then $\cs_{r}(G;\E)$ is non-commutative. To see this, assume first that there exists $e\in \E$ with $r(e)\neq s(e)$. Let $h\in C_0(G\z)$ with $h(r(e))=1$ and $h(s(e))=0$ and $f\in C_c(G;\E)$ with $f(e)\neq 0$ and $\pi(\suppo(f))$ a bisection (Lemma~\ref{lem:Urysohn}). Then
  \[
  (h\ast f)(e)  
    =h(r(e))f(e)=
  f(e)
  \neq
  0
  =f(e)h(s(e))=
    (f\ast h)(e),
  \]
  so $\cs_{r}(G;\E)$ is not commutative.
  
  Next, assume $e_{1},e_{2}\in \E(u)$ for some $u\in G\z$ are such that $e_{1}e_{2}\neq e_{2}e_{1}$. Take elements $f_{i}\in C_c(G;\E)$ with $\pi(\suppo(f_{i}))$ a bisection and $f_{i}(e_{i}) \neq 0$, so that
  \[
    (f_{1}\ast f_{2})(e_{1}e_{2}) = f_{1}(e_{1})f_{2}(e_{2}) \neq 0.
  \]
  On the other hand, 
  \[
    (f_{2}\ast f_{1})(e_{1}e_{2}) = 
    \sum_{\set{\pi(e)\in G:r(e)=u}} f_{2}(e)f_{1}(e\inv e_{1}e_{2}).
  \]
  By assumption, $\pi(e_{2})$ is the unique element of $G_u\cap \pi(\suppo(f_{2}))$, so that
  \[
      (f_{2}\ast f_{1})(e_{1}e_{2})
      =
      f_{2}(e_{2})f_{1}(e_{2}\inv e_{1}e_{2}).
  \]
  If $\pi(e_{2}\inv e_{1}e_{2})\neq \pi(e_{1})$, then $f(e_{2}\inv e_{1}e_{2})=0$ and thus
  \[
    (f_{2}\ast f_{1})(e_{1}e_{2})
    =
    0
    \neq
    (f_{1}\ast f_{2})(e_{1}e_{2}).
  \]
  If $\pi(e_{2}\inv e_{1}e_{2})= \pi(e_{1})$, then by the assumption that $e_{1}e_{2}\neq e_{2}e_{1}$,  there exists $z\in\T \setminus\set{1}$ such that $e_{2}\inv e_{1}e_{2}=z\cdot e_{1}$, so that
  \[
    (f_{2}\ast f_{1})(e_{1}e_{2})
    =
    f_{2}(e_{2})f_{1}(z\cdot e_{1})
    =
    z f_{1}(e_{1}) f_{2}(e_{2})
    \neq
    f_{1}(e_{1}) f_{2}(e_{2})
    =
    (f_{1}\ast f_{2})(e_{1}e_{2}).
  \]
    In both cases, we conclude $f_{1}\ast f_{2}\neq f_{2}\ast f_{1}$.
\end{example}

The following is a straightforward variation on
\cite{Sims:gpds}*{Lemma~9.1.3}.
\begin{lemma}
  \label{lem-bisec-span} The space $C_{c}(G;\E)$ is spanned by
  functions whose support is the preimage under $\pi$ of a bisection
  of~$G$.
\end{lemma}

\begin{proof}
  Let $f\in C_{c}(G; \E)$.  Since $G$ is \etale, it has a basis of
  open bisections for its topology
  \cite{Renault:gpd-approach}*{Proposition~I.2.8}.  Since $\supp(f)$
  is compact, $\pi(\supp(f))$ can be covered with finitely many
  precompact open bisections $U_{1},\dots, U_{n}$.  Let
  $\sset{\phi_{i}}_{i=1}^{n}$ be a partition of unity in
  $C_{c}^{+}(G)$ subordinate as in \cite{wil:crossed}*{Lemma~1.43} for
  $\sset{U_{i}}_{i=1}^{n}$.  Then
  $f_{i}=(\phi_{i}\circ\pi)\cdot f\in C_{c}(G;\E)$ has
  $\supp f_{i} \subset \pi\inv(U_{i})$ and satisfies
  $f=\sum_{i} f_{i}$.
\end{proof}

The following generalization of Renault's
\cite{Renault:gpd-approach}*{Proposition~II.4.2} will be used
frequently.

\begin{prop}
  [\cite{BFPR:GammaCartan}*{Proposition~2.8}] \label{prop-bfpr}Let
  $\E$ be a twist over an \etale\ grou\-poid $G$.  Then there is a
  norm-decreasing linear map
  $j_{G}\colon \cs_{r}(G;\E)\to C_{0}(G;\E)$ such that
  $j_{G}(f)(e)=f(e)$ if $f\in C_{c}(G;\E)$.  Furthermore, for all
  $a,b\in \cs_{r}(G;\E)$ we have
  \begin{equation}
    \label{eq:7}
    j_{G}(a^{*})(e)=\overline{j_{G}(a)(e\inv )} \quad\text{and} \quad
    j_{G}(ab)(e') =\sum_{\set{\pi(e)\in G:r(e)=r(e')}} j_{G}(a)(e)
    j_{G}(b)(e\inv e') .
  \end{equation}
\end{prop}

\begin{remark}
  \label{rem-con-defined} One consequence of
  Proposition~\ref{prop-bfpr} is that the convolution product
  \eqref{eq:2} makes sense for all
  $f,g\in j_{G}\bigl(\cs_{r}(G;\E) \bigr)\subset C_{0}(G;\E)$.  We
  don't know if it can be extended to all of $C_{0}(G;\E)$.  However,
  provided either $f$ or $g$ is in $C_{c}(G;\E)$, it is easy to see
  that the formula certainly is well-defined---in fact the sum is
  finite.
\end{remark}

\begin{lemma} [\cite{BFPR:GammaCartan}*{Lemma~2.7}]
  \label{lem:j-map commutes} Let $\E$ be a twist over an \etale\
  groupoid $G$.  If $S$ is an open subgroupoid of~$G$ and
  we let
  $\E_{S}=\pi\inv (S)$, then
  \begin{equation}
    \label{eq:8}
    \begin{tikzcd}
      S\z\times \T\arrow[r,"\iota",hook] &\E_{S} \arrow[r,"\pi",two
      heads]& S
    \end{tikzcd}
  \end{equation}
  is a twist over $S$ and ``extension by zero'' from
  $C_{c}(S;\E_{S})\to C_{c}(G;\E )$ extends to an inclusion
  $i\colon \cs_{r}(S;\E_{S})\to \cs_{r}(G;\E)$.  Furthermore, the
  diagram
  \begin{equation}
    \label{eq:9}
    \begin{tikzcd}
      C_{c}(S;\E_{S}) \arrow[r,hook] \arrow[d,hook,"i"] &
      \cs_{r}(S;\E_{S}) \arrow[r,,"j_{S}"] \arrow[d,"i",hook] &
      C_{0}(S;\E_{S}) \arrow[d,hook,"i"] \\
      C_{c}(G;\E) \arrow[r,hook] & \cs_{r}(G;\E) \arrow[r,"j_{G}"] &
      C_{0}(G;\E)
    \end{tikzcd}
  \end{equation}
  commutes.
\end{lemma}

\begin{proof}
  All but the commutativity of the diagram is in
  \cite{BFPR:GammaCartan}*{Lemma~2.7}.  Since the diagram clearly
  commutes on the level of $C_{c}$-functions, it commutes in general
  by continuity.
\end{proof}

As described in the introduction, we aim to find Cartan subalgebras of
$\cs_r(G;\E)$.  For easy reference, we reall the definition here.

\begin{definition}
  [\cite{ren:irms08}*{Definition~5.1}] \label{def-cartan}A
  \cs-subalgebra $B$ of a \cs-algebra $A$ is called \emph{Cartan} if
  \begin{enumerate}
  \item $B$ is maximal abelian,
  \item $B$ is regular, meaning that the set
    \[N(B)=\set{n\in A:
     \text{$nBn^{*}\subset B$ and $n^{*}Bn\subset B$}
        }\] of \emph{normalizers} of $B$ generates
    $A$ as a \cs-algebra, and
  \item there is a faithful conditional expectation
    $\Phi\colon A\to B$.
  \end{enumerate}
\end{definition}

\begin{remark}
  Originally, 
  the definition of a Cartan subalgebra
  contained the additional
  requirement that $B$ should contain an approximate identity for $A$.
  Recently, Pitts has shown this to be superfluous
  \cite{pitts2021normalizers}.
\end{remark}

Before we focus on our main results, let us prove the following lemma which will come in handy multiple times.
\begin{lemma}\label{lem:int is sbgpd}
    Suppose $G$ is a \LCH\  groupoid with open range map, and let $\E$ be a twist over $G$ as in \eqref{eq:50}.
    \begin{enumerate}[label=\textup{(\alph*)}]
    \item\label{it:int(H)} If $H$ is a subgroupoid of $G$, then so is $\Int[G]{H}$, and $\Int[\E]{\pi\inv(H)}=\E_{\Int[G]{H}}$.
    \item\label{item:saturation still abelian} If $\mathcal{F}$ is a
      subgroupoid of $\E$, then so is $\pi\inv (\pi(\mathcal{F}))$. If
      $\mathcal{F}$ is abelian, then so is $\pi\inv (\pi(\mathcal{F}))$.
    \item\label{item:int of IsoE vs int of IsoG}
      $\Int[\E]{\Iso{\E}}=
      \E_{\Int[G]{\Iso{G}}}
      $ and
      $\pi(\Int[\E]{\Iso{\E}})=\Int[G]{\Iso{G}}$. In particular,
      $\Int[\E]{\Iso{\E}}$ is a subgroupoid of~$\E$.
    \end{enumerate}
\end{lemma}
\begin{proof}
\ref{it:int(H)}
As the inversion map of  $G$ is a homeomorphism, it restricts to a map on $H^\circ \coloneq \Int[G]{H}$.
 Since the range map of $G$ is open,
  its multiplication map is open by \cite{Sims:gpds}*{Lemma 8.4.11}, so it restricts to a multiplication map on $H^\circ$, proving that $H^\circ$ is a subgroupoid of~$G$.
  Since $\pi$ is continuous, we have
  \[
    \E_{H^\circ}
    =
    \pi\inv (H^\circ)
    \subset
    \Int[\E]{\pi\inv(H)}
    .
  \]
  For the other inclusion, note that $\pi(\pi\inv(H))=H$ by surjectivity of $\pi$, so since $\pi$ is an open map, we have
  \[
    \pi(\Int[\E]{\pi\inv(H)}) \subset H^\circ,
  \]
  which implies
  \[
    \Int[\E]{\pi\inv(H)}
    =:
    V
    \subset 
    \pi\inv(\pi(V)) \subset \pi\inv (H^\circ)
    =
    \E_{H^\circ},
  \]
  as claimed.

    \ref{item:saturation still abelian} Because $\pi$ is a groupoid
    homomorphism, $\pi(\mathcal{F})$ is a subgroupoid of~$G$. Therefore,
    $\pi\inv (\pi(\mathcal{F}))$ is a subgroupoid of $\E$ by
    \cite{AMP:IsoThmsGpds}*{Proposition 9}. Now assume that $\mathcal{F}$ is
    abelian. For any $e \in \pi\inv (\pi(\mathcal{F}))$, there exists
    $e' \in \mathcal{F}$ such that $\pi(e) = \pi(e')$.  Since $\mathcal{F}$ is
    abelian, we have that $r(e)=r(e')$ coincides with $s(e)=s(e')$,
    proving that $\pi\inv (\pi(\mathcal{F})) \subset \Iso{\E}$. To see that
    $\pi\inv (\pi(\mathcal{F}))$ is abelian, fix
    $e_1,e_2 \in \pi\inv (\pi(\mathcal{F}))$.  Then there exists
    $e_1',e_2' \in \mathcal{F}$ and $z_1,z_2 \in \T$ such that
    $e_1 = z_1\cdot e_1'$ and $e_2 = z_2\cdot e_2'$.  Since $\E$ is a
    central extension
    \begin{align*}
      e_1e_2 &= (z_1 \cdot e_1')(  z_2\cdot e_2') = (z_1z_2)\cdot (e_1' e_2')
      \\
      \intertext{which, since $\mathcal{F}$ is abelian, is}
             &= (z_2z_1)\cdot (e_2' e_1') \\
             &= (z_2 \cdot e_2')( z_1 \cdot e_1') = e_2 e_1,
    \end{align*}
    so $\pi\inv (\pi(\mathcal{F}))$ is abelian.
    
    \ref{item:int of IsoE vs int of IsoG} 
    The claim that $\Int[\E]{\Iso{\E}}=\E_{\Int[G]{\Iso{G}}}$ follows from Part~\ref{it:int(H)} applied to $H=\Iso{G}$ and from the fact that $\Iso{\E}=\pi\inv(\Iso{G})$ since $\pi$ preserves ranges and sources.
    The
    equality $\pi(\Int[\E]{\Iso{\E}})=\Int[G]{\Iso{G}}$ now follows
    from surjectivity of $\pi$. 
    By Part~\ref{it:int(H)}, $\Int[G]{\Iso{G}}$ is a subgroupoid of~$G$
    and so
    \( \Int[\E]{\Iso{\E}}=\pi\inv(\Int[G]{\Iso{G}}) \) is a
    subgroupoid of $\E$, again by \cite{AMP:IsoThmsGpds}*{Proposition
    9}.
\end{proof}

\section{The Main Theorem}
\label{sec:main-theorem}

Our main theorem is a generalization of \cite{DGNRW:Cartan}*{Theorem
3.1} to general twisted groupoid \cs-algebras.  We first need the following definition:
\begin{definition}\label{def:Ad_S,Omega}
    Assume $\E$ is a twist as in \eqref{eq:50} over a \LCH\ groupoid $G$. Assume $S$ is an open subgroupoid of $G$. For $u\in \go$ and $e\in \E(u)$, write
    \begin{equation}
        \operatorname{Ad}_{S}(e)
        \coloneq
        \set{
            \tau\inv e \tau : \tau\in \E_{S}(u)
        },
    \end{equation}
    with the understanding that
    $\operatorname{Ad}_{S}(e)=\emptyset$ if $u\notin S\z$.
    Note that the set $\pi(\operatorname{Ad}_{S}(e))$ does not depend on $e$ but only on $\pi(e)$, so we may let
    \[
        \operatorname{Ad}_{S}(\pi(e))\coloneq \pi(\operatorname{Ad}_{S}(e))
        .
    \]
    Let $\overline{\N}=\Z_{\geq 0}\cup\set{\infty}$ denote the extended natural numbers, and define $\Omega_{S}\colon \Iso{\E}\to\overline{\N}$ for $e\in \Iso{\E}$ by
  \[
    \Omega_{S}(e) =
    \begin{cases}
        \left|\operatorname{Ad}_{S}(e)\right|,&\text{ if }  \left| \operatorname{Ad}_{S}(e)\right| = \left|\operatorname{Ad}_{S}(\pi(e))\right|<\infty,        
        \\
        \infty&\text{ otherwise.}
    \end{cases}
  \]
  In other words, $\Omega_S(e)$ checks whether or not $\pi$ is injective on the set $\operatorname{Ad}_{S}(e)$, and whether $\operatorname{Ad}_{S}(\pi(e))$ is finite.
\end{definition}

\begin{remark}
In the case where $\E=\E_c$ is given by a $2$-cocycle $c$, an equality of the form $(t,w)\inv (g,z) (t,w)=(g,z)$ in $\E_{c}$ implies not only that $t\inv gt=g$ but also that $c(g,t)=c(t,g)$. It is easy to see in examples
that the condition $c(g,t)=c(t,g)$ is \emph{not} implied by $t\inv g t=g$.
In other words, the equality $\tau\inv e\tau = e$ in 
any twist
$\E$ is a stronger condition than $\pi(\tau)\inv \pi(e) \pi(\tau) = \pi(e)$ in $G$, and we can 
in general
only deduce $|\operatorname{Ad}_{S}(e)|\geq|\operatorname{Ad}_{S}(\pi(e))|$.
In Example~\ref{ex:rotation}, the reader can find an example in which this inequality is strict, since there, the entire groupoid $G$ is abelian, but the twist is not.
\end{remark}

In the following, $i$
denotes the inclusion map from Lemma~\ref{lem:j-map commutes}.
\begin{thm}\label{thm:May17}
   Suppose 
  $\E$ is a twist 
  over a \LCH, \etale\ groupoid~$G$, and $S$ is an open subgroupoid of $G$. Let $B=i(\Cst_{r}(S;\E_{S}))$ and $A=\Cst_{r}(G;\E)$. Then the following statements  are equivalent.
  \begin{enumerate}[label=\textup{(\roman*)}]
    
    \item\label{it:thm:May17:B:subset} $B$ is a Cartan subalgebra of $A$ for which
    \begin{equation}
        \set{
        h\in C_{c}(G;\E):
        \pi(\suppo(h)) \text{ is a bisection}
        } \subset N(B).
    \end{equation}
    
    \item\label{it:thm:May17:B:cap} $B$ is a Cartan subalgebra of $A$ for which the set
    \begin{equation}\label{eq:thm:May17:B cap} 
       \set{
        a\in A:
        \pi(\suppo(j_{G}(a))) \text{ is a bisection}
        } 
        \cap  N(B)
    \end{equation}
    generates $A$.
    
    \item\label{it:thm:May17:S:Omega} We have 
      \begin{gather}  
      \go\subset S
      \text{ and }
      \Int[\E]{\Omega_{S}\inv(\set{1})}=\E_{S},
      \label{eq:thm:May17:Int Omega=1}
      \tag{%
      \textsf{max}}
      \\
        \label{eq:thm:May17:icc}
        \Int[\E]{\Omega_{S}\inv (\Z_{>1})}=\emptyset
        \tag{%
        \textsf{ricc}}
        ,
    \end{gather}
      and $S$ is closed and normal in $G$.
      \item\label{it:thm:May17:S:Omega+max} $S$ is maximal among open subgroupoids of $\Int[G]{\Iso{G}}$ for which
      $\E_{S}$ is abelian,  $S$ satisfies Condition~\eqref{eq:thm:May17:icc}, and $S$ is closed and normal in $G$.
  \end{enumerate}
\end{thm}

Note that the equivalence  \ref{it:thm:May17:B:subset}$\iff$\ref{it:thm:May17:S:Omega+max} in Theorem~\ref{thm:May17} is exactly the content of Theorem~\ref{thmA}.
Since Conditions~\eqref{eq:thm:May17:Int Omega=1} and~\eqref{eq:thm:May17:icc} might be difficult to parse, Section~\ref{sec:applications} contains  corollaries in which we list stronger, easier to understand properties of~$S$ that imply that $i(\Cst_{r}(S;\E_{S}))$ is a Cartan subalgebra.
Before embarking on the proof of the theorem, let us make some remarks. 

\begin{remark}
    By definition of $\Omega_{S}$, Condition~\eqref{eq:thm:May17:Int Omega=1} implies that $\E_{S}\subset \Iso{\E}$, so that $S\subset \Iso{G}$. In the setting of Item~\ref{it:thm:May17:S:Omega}, it therefore makes sense to ask for $S$ to be normal. 
\end{remark}

\begin{remark}\label{rmk:intG subgpd}
  It follows from Lemma~\ref{lem:int is sbgpd}\ref{it:int(H)}
  that $\Int[G]{\Iso{G}}$ is a subgroupoid of~$G$. Note further that, since $\pi$ is a groupoid homomorphism, $\E_{S}$ is
  subgroupoid of~$\E$. 
  Therefore it is reasonable to
  talk about $S$ being a maximal open subgroupoid \emph{of the interior of  $\Iso{G}$} for
  which $\E_{S}$ is abelian, as we have done in
  Item~\ref{it:thm:May17:S:Omega+max}.
\end{remark}

\begin{remark}[Regarding $\Omega_{S}$]
    The `gate keeper' $\Omega_{S}$ is assessing whether or not $\Iso{\E}$ has points $e$ that stand in the way of $B\coloneq i(\Cst_{r}(S;\E_{S}))$ being maximal abelian in $A\coloneq\Cst_{r} (G;\E)$ (see Proposition~\ref{prop:max_new}).

    By comparing the statements in~\ref{it:thm:May17:S:Omega} and \ref{it:thm:May17:S:Omega+max} of Theorem~\ref{thm:May17}, one can see that 
    Condition~\eqref{eq:thm:May17:Int Omega=1} 
    ensures that $B$ is maximal among the abelian subalgebras of $A$ that arise from open subgroupoids
    (see Corollary~\ref{cor:S max in terms of Omega})%
    , hence the name of the condition.
    On the other hand, in Condition~\eqref{eq:thm:May17:icc}, the map $\Omega_{S}$ discerns whether there are points $e\in\Iso{\E}$ for which
    \[
    1 < |\{\sigma\inv e\sigma:\sigma \in \E_{S}\}|= |\{s\inv \pi(e)s:s \in S\}|<\infty;
    \]
    i.e., points $e$ that do not have {\em relative infinite conjugacy class} (relative with respect to the subgroupoid $S$). The
    proof of Proposition~\ref{prop:max_new}  
    shows why an open set of such points stands
     in the way of $B$ being maximal
     (the key is the technical result in Lemma~\ref{lem:fancy lemma}).
     Let us briefly explain here why those points for which $\Omega_S(e)$ takes the value $\infty$, do not `pose a threat' to the maximality of $B$.
   
    There are two scenarios in which $\Omega_S(e)=\infty$. In the first, $\operatorname{Ad}_{S}(e)$  has cardinality larger than that of $\operatorname{Ad}_{S}(\pi (e ))$, so
    there exists $\sigma\in\E_{S}$ such that $\sigma\inv e\sigma \neq e$ but $\pi(\sigma\inv e\sigma)=\pi(e)$, i.e., $\sigma\inv e\sigma = z\cdot e$ for some $z\in \T\setminus\{1\}$. With this condition, any $h\in 
    C_c
    (G;\E)$ with $e\in \suppo(h)$ satisfies
    \[
    0
    \neq
    zh(e) = h(z\cdot e) =  h(\sigma\inv e\sigma).
    \]
    As $z\neq 1$, this proves $h(e)\neq h(\sigma\inv e\sigma)$. We will see in Lemma~\ref{lem:conjugation equality} that $h$ therefore does not commute with 
    $B$.
    In other words, there is no element of $A$ that both commutes with $B$ and does not vanish at $e$, and
     so the existence of $e$
    cannot be the cause of $B$ not being maximal abelian in $A$.

    The second scenario in which $\Omega_S(e)$ takes on the value $\infty$, is when $\operatorname{Ad}_S(\pi(e))$ (and hence $\operatorname{Ad}_S(e)$) is infinite. If $h\in
    C_c
    (G;\E)$ commutes with $
    B
    $, then the reverse implication of Lemma~\ref{lem:conjugation equality} states that $h(e)=h(e')$ for every $e'\in \operatorname{Ad}_S(e)$. But since the isotropy groups of $G$ are discrete and since  $\operatorname{Ad}_S(\pi(e))$ is infinite, this implies that $h(e)=0$. 
    So
    again, $e$ cannot be the cause of $B$ not being maximal abelian.
\end{remark}

\begin{remark}
    The relevance of the cardinality of~$S$-conjugates actually precedes even
  the results in \cite{DGNRW:Cartan}; for example, it can be found in
  \cite{BrOz:FinDimApp}*{p.\ 360} in the proof that the group von
  Neumann algebra of $\Z^2$ is maximal abelian in
  $L(\Z^2\rtimes \operatorname{SL}_{2}(\Z))$.
\end{remark}

\begin{remark}
Let $S$ and $G$ be as specified in Theorem~\ref{thm:May17}\ref{it:thm:May17:S:Omega+max}.
    According to \cite{ren:irms08}, the Cartan pair $i(\Cst_{r}(S;\E_{S})) \subset \Cst_{r}(G;\E)$  gives rise to a Weyl groupoid and Weyl twist, and one may wonder how it relates to the pair $S\subset G$.
    
    In the case that $\E$ is induced by a $2$-cocycle,  \cite{DGN:Weyl} identified the effective Weyl groupoid as a certain transformation groupoid and, assuming that the quotient map $G\to G/S$ allows a continuous section, the Weyl twist was identified as a certain $2$-cocycle. The recent pre-print \cite{BG:2023:Gamma-Cartan-pp} extends those $2$-cocycle results to the setting of $\Gamma$-Cartan pairs from subgroupoids of $\Gamma$-graded groupoids. 

    It is somewhat straight forward to prove that, in the setting of
    Theorem~\ref{thm:May17}\ref{it:thm:May17:S:Omega+max},
    the Weyl groupoid can likewise be identified as a transformation groupoid of a continuous action by the quotient groupoid $G/S$ on the Gelfand dual of $i(\Cst_{r}(S;\E_{S})) $. Moreover, this dual can be described in terms of $\E_{S}$,  along the same lines of what is done in \cite{DGN:Weyl}*{Section 3} 
    in the $2$-cocycle case.
    The Weyl twist, on the other hand, is more elusive. We deal with a detailed description of both the Weyl groupoid and the Weyl twist in a separate paper
    \cite{D:2025:DGNvsDWZ-inprep}; the interested reader can also find a related result in \cite{Renault:2023:Ext}*{Theorem 5.5}. In \cite{D:2025:DGNvsDWZ-inprep}, we also answer the question when a normal subgroupoid gives rise to a $\cs$-diagonal, i.e., a Cartan subalgebra with the unique extension property.
\end{remark}

\begin{remark}
    If the twist $\E$ is induced by a $2$-cocycle $c$ (see Example~\ref{ex:twist given by $2$-cocycle}), one could of course translate the equivalent conditions of
    Theorem~\ref{thm:May17}
    into conditions on $c$. As we have not found these conditions to be very enlightening, we refrain from doing so here. 
\end{remark}

Comparing 
Theorem~\ref{thm:May17}
to Renault's correspondence between Cartan pairs of $\Cst$-algebras and effective groupoids with twist (\cite{ren:irms08}*{Theorems 5.2 and 5.9}), has led us to notice the following alternative characterization of effectiveness.

   \begin{lemma}\label{lem:effectiveness}
       Suppose $G$ is a \LCH, \etale\ groupoid. The following are equivalent.
       \begin{enumerate}[label=\textup{(\roman*)}]
           \item $G$ is effective;
           \item $G$ contains no non-trivial open abelian subgroupoid.
       \end{enumerate}
   \end{lemma}

   \begin{proof}
    If a subgroupoid $S$ is abelian, then it is contained in $\Iso{G}$. If it is open, it is contained in $\Int[G]{\Iso{G}}$. If $G$ is effective, the latter equals $G\z$, meaning $S=r(S)$ is trivial. 
    
    Conversely, assume $G$ 
    is non-effective. Since $G$ is \etale, this means we can find a non-empty open bisection $V\subset  G\setminus G\z$ such that $V\subset  \Iso{G}$. Since $V$ is a bisection, we have $V=\bigsqcup_{u\in r(V)} \{g_u\}$, where $V\cap G(u)=\{g_u\}$. Since the $g_u$'s do not interact with one another, the  subgroupoid 
     $\langle V\rangle$ of $G$ generated by $V$ is exactly the bundle $\bigsqcup_{u\in r(V)}\langle g_u\rangle$ of abelian groups generated by the individual $g_u$, so $\langle V\rangle$ is an abelian subgroupoid of $\Iso{G}$. Since $V$ is open, $S\coloneq\Int[G]{\langle V\rangle}$ contains $V$.
  By
  Lemma~\ref{lem:int is sbgpd}\ref{it:int(H)},
  $S$ is an open subgroupoid of $G$; as $\langle V\rangle$ is abelian, so is $S$; and as $\emptyset \neq V\subset  S\cap (G\setminus G\z)$, $S$ is non-trivial.
   \end{proof}

\subsection*{Standing Assumptions} 
For the remainder of this paper, $G$ will denote a \LCH, \etale\ groupoid; $S$ is an open subgroupoid of~$G$; and $\E$ is a twist over $G$ as in \eqref{eq:50}. We will invoke other hypotheses only as
necessary.

\subsection{The conditional expectation}

\begin{prop}[{%
cf.\ \cite{BEFPR:2021:Intermediate}*{Lemma 3.4}%
}]\label{prop:CondExp}
  Suppose  
  the subgroupoid 
 $S$ is not only open but also closed
 in $G$.
 Define
  $\Phi_0\colon C_c(G;\E) \to C_c(S;\E_{S})$ by
  $ \Phi_0(f) \coloneq f\restr{\E_{S}}$. 
  \begin{enumerate}[label=\textup{(\arabic*)}]
  \item\label{it:CondExp:Phi_0 extends} $\Phi_0$ extends to a linear map
    $ \Phi\colon \Cst_{r}(G;\E) \to \Cst_{r}(S;\E_{S})$.
  \item\label{it:CondExp:i circ Phi} The map $ i \circ\Phi$ is a conditional expectation from
    $\Cst_{r}(G;\E)$ to $ i (\Cst_{r}(S;\E_{S}))$.
  \item\label{it:CondExp:j-maps} We have
    $j_{S}(\Phi(a)) = j_G(a)\restr{\E_{S}}$ for all
    $a \in \Cst_{r}(G;\E)$.
  \item\label{it:CondExp:faithful} If $S\z=\go$, then $\Phi(a^*a)=0$ implies $a=0$, so
    $ i \circ\Phi$ is faithful.
  \end{enumerate}
\end{prop}
\begin{proof}
Both~\ref{it:CondExp:Phi_0 extends} and~\ref{it:CondExp:i circ Phi} are contained in \cite{BEFPR:2021:Intermediate}*{Lemma 3.4}.
For completion, we add a proof here. We will follow the ideas of the proof of \cite{DGNRW:Cartan}*{Proposition 3.13}. 

\ref{it:CondExp:Phi_0 extends} Because $S$ is closed and $\pi$ is continuous,
  $\pi\inv(S) = \E_{S}$ is closed, so that $\Phi_{0}$ is well defined.
  For $u\in \go$, let $H^{G}_{u}\coloneq \ell^2(G_{u};\E_{u})$ and
  $\theta^{G}_{u}\colon C_c(G;\E) \to B\left(H^{G}_{u}\right)$, the
  regular representations from~\eqref{eq:5} for $\E$; we analogously
  define $H^{S}_{u}\coloneq \ell^2(S_{u};(\E_{S})_{u})$ and
  $\theta^{S}_{u}$ for $u\in S\z$.  Let
  $P \in B\left(H^{G}_{u}\right)$ be the projection onto the subspace
  of elements that vanish outside of the (potentially empty) subset
  $\E_{S}\cap \E_{u}$. Note that there is a canonical unitary
  $F\colon H^{S}_{u} \to PH^{G}_{u}$. For $f\in C_c(G;\E)$ and
  $\xi\in H^{S}_{u}$, we claim that
  \begin{equation}\label{eq:theta-G and theta-S}
    P \theta^{G}_{u}(f) (F(\xi)) = F\left(
      \theta^{S}_{u}(f\restr{\E_{S}})(\xi)\right). 
  \end{equation}
  By construction, both sides are elements of $PH^{G}_{u}$, and so
  they vanish outside of $\E_{S}$. At $e\in \E_{S}$, we have
  \begin{align*} [P \theta^{G}_{u}(f) (F(\xi))](e) &=
    \sum_{\set{\pi(e_{1})\in G: r(e)=r(e_{1})}} f(ee_{1}\inv)\,
    F(\xi)(e_{1})
    \\
       &=
         \sum_{\set{\pi(e_{1})\in
         S:
         r(e)=r(e_{1})}}
         f(ee_{1}\inv)\,
         \xi(e_{1})
    \\
       &\overset{(\dagger)}{=}
         \sum_{\set{\pi(e_{1})\in
         G:
         r(e)=r(e_{1})}}
         f\restr{\E_{S}}(ee_{1}\inv)\,
         \xi(e_{1})
    \\
       &= [F\left(
         \theta^{S}_{u}(f\restr{\E_{S}})(\xi)\right)](e). 
  \end{align*}
  Here, $(\dagger)$ follows from the fact that $\pi(e_{1})\in S$ if
  and only if $ee_{1}\inv\in \E_{S}$, since $e\in \E_{S}$ and since
  $\E_{S}$ is a subgroupoid. 
  This proves~\eqref{eq:theta-G and theta-S}. It 
  follows that
\begin{align*}
    \norm{\theta^{S}_{u} (f\restr{\E_{S}})}
    &=
    \sup\left\{\norm{F\left(
      \theta^{S}_{u}(f\restr{\E_{S}})(\xi)\right)} : \xi\in H^{S}_{u},
      \norm{\xi}=1\right\}
      &&
      \text{since $F$ is a unitary}
    \\
    &\leq
    \sup\left\{\norm{\theta^{G}_{u}(f) (F(\xi))} : \xi\in H^{S}_{u},
      \norm{\xi}=1\right\}
      &&
      \text{by \eqref{eq:theta-G and theta-S}}
      \\
    &\leq
    \sup\left\{\norm{\theta^{G}_{u}(f) (\eta)} : \eta\in H^{G}_{u},
      \norm{\eta}=1\right\} 
    =
    \norm{\theta^{G}_{u}(f)}.
\end{align*}
Since $u\in S\z\subset  \go$ was arbitrary, this shows that
$\norm{\Phi_{0}(f)}_{r}\leq \norm{f}_{r}$. 
Thus, $\|\Phi_0\| \leq 1$, so that $\Phi_0$ extends to a continuous
linear map $\Phi\colon \Cst_{r}(G;\E) \to
\Cst_{r}(S;\E_{S})$.

\ref{it:CondExp:i circ Phi}
Clearly, $\Phi_0\circ i $ is the identity on $C_c(S;\E_{S})$, so
that $ i \circ \Phi\circ i = i $. It follows that
$\mathmbox{i\circ\Phi}
\colon \Cst_r(G;\E)\to\Cst_r(G;\E)$ maps onto and
fixes the subalgebra $ i (\Cst_r(S;\E_{S}))$ of
$\Cst_r(G;\E)$. Since $ i \circ\Phi$ is furthermore contractive
and linear, it follows from \cite{BrOz:FinDimApp}*{Theorem 1.5.10}
that it is a conditional expectation.
       
\ref{it:CondExp:j-maps}
To see that $j_{S}(\Phi(a)) = j_G(a)\restr{\E_{S}}$ for all
$a \in \Cst_{r}(G;\E)$, note that the claim is trivial for
$a\in C_c(G;\E)$: in this case, $\Phi(a)=a\restr{\E_{S}}$, and the
maps $j_{G}$ and $j_{S}$ are just the inclusions of $C_c(G;\E)$
and $C_c(S;\E_{S})$ into $C_0(G;\E)$ resp.\
$C_0(S;\E_{S})$. Continuity of all maps involved implies the
claim
for arbitrary $a$.

\ref{it:CondExp:faithful}
Assume $S\z=\go$. To see that $\Phi$ is faithful, recall that
\cite{ren:irms08}*{Proposition 4.3} says that the map
$C_c(G;\E)\to C_0 (\go), f\mapsto f\restr{\go}$, extends to a faithful
conditional expectation $P_{G}\colon \Cst_r(G;\E)\to C_0 (\go)$.
The  diagram
    \[
    \begin{tikzcd}
        \Cst_r(G;\E) \ar[rr, "P_{G}"] \ar[rd, "\Phi"'] && C_0 (\go)
        \\
        & \Cst_{r}(S;\E_{S}) \ar[ru, "P_{S}"'] &
    \end{tikzcd}
    \]
    commutes since it commutes on the dense subspace $C_{c}
    (G;\E)$.
    If $\Phi(a^*a)=0$, then commutativity implies $P_{G}(a^*a)=0$, so
    $a=0$ since $P_{G}$ is faithful. 
    \end{proof}

    The following lemma explains under which circumstances an element
    of the large algebra $\Cst _r (G;\E)$ is actually in the
    subalgebra $ i (\Cst_r(S;\E_{S}))$. It is a generalization of
    \cite{DGNRW:Cartan}*{Lemma 3.8}, and it should also be compared to
    spectral theorems for modules such as \cite{MS:1989:Subalgs}*{Theorem 3.10}, \cite{Qiu:1991:Triangular}*{p.\ 4},
    \cite{Hopenwasser:2005:Spectral}.

\begin{prop}[cf.\ {\cite{DGNRW:Cartan}*{Lemma
    3.8}}]\label{prop:support in S} 
  We have 
  \[ i (\Cst_{r}(S;\E_{S})) \subset  \set{b \in \Cst_{r}(G;\E):
  \suppo(j_G(b)) \subset  \E_{S}}.
  \]
  If $S$ is not only open but also closed, then we have equality.
\end{prop}

\begin{remark}\label{rmk:G amenable}
  If $G$ is amenable, then the assumption in Proposition~\ref{prop:support in S} that $S$ be closed is not needed to conclude equality of the sets; see
  \cite{BEFPR:2024:Corrigendum}.
     Our proof, however, heavily
  relies on the conditional expectation
  constructed in Proposition~\ref{prop:CondExp} which, by \cite{BEFPR:2021:Intermediate}*{Lemma 3.4},
  exists if and only if $S$ is closed.
\end{remark}

\begin{proof}
  First, note that
  $B\coloneq\set{b \in \Cst_{r}(G;\E): \suppo(j_{G}(b)) \subset 
  \E_{S}}$ is closed: if $\sset{b_n}$ is a sequence in $B$
  converging to $b \in \Cst_{r}(G;\E)$, and if
  $e \in
  \E\setminus\E_{S}$, 
  then continuity of $j_{G}$ implies that
  $0=j_G(b_n)(e)\to j_G(b)(e)$, so $j_G(b)(e) = 0$.

  For $f \in C_c(S; \E_{S})$,
  $\suppo(j_{G}(i(f))) = \suppo(f)$ is a subset of $\E_{S}$, so
  $i(C_c(S;\E_{S})) \subset  B$. Since $B$ is closed, since
  $i$ is continuous, and since $C_c(S;\E_{S})$ is dense in
  $ \Cst_{r}(S;\E_{S})$, we conclude
  $i(\Cst_{r}(S;\E_{S})) \subset  B$.

  Now fix $b \in B$
  and assume that $S$ is clopen, so that
  we can consider the (not
  necessarily faithful) map
  $ \Phi\colon \Cst_{r}(G;\E) \to \Cst_{r}(S;\E_{S})$ from
  Proposition~\ref{prop:CondExp}. If
  $i'\colon C_0(S;\E_{S})\to C_0(G;\E)$ denotes the
  inclusion as in Lemma~\ref{lem:j-map commutes},
  then 
  we have
\begin{align*}
   j_G(b)
   &=
   i'(j_G(b)\restr{\E_{S}})
   &&\text{since $\suppo(j_G(b))\subset  \E_{S}$}
   \\
   &= i'(j_{S}(\Phi(b)))
   &&\text{by Proposition~\ref{prop:CondExp}\ref{it:CondExp:j-maps}} 
   \\ 
   &= j_G(i(\Phi(b)))
   &&\text{by Lemma~\ref{lem:j-map commutes}}.
\end{align*}
Since $j_G$ is injective, we conclude $b = i(\Phi(b))\in
i(\Cst_{r}(S;\E_{S}))$.    
\end{proof}

\subsection{Maximality of $B$}

Our next focus is maximality of $B$.
The main result of this section is the following proposition. 
The ideas of the proof generalize those found in \cite{DGNRW:Cartan}*{Section 7}.

\begin{prop}\label{prop:max_new}
  Assume that
  $S$ is not only open 
  but also closed
  in $G$.
  Then the following statements
  are equivalent.
  \begin{enumerate}[label=\textup{(\roman*)}]
  \item\label{it:prop:max_new:B} $i(\Cst_{r}(S;\E_{S}))$ is maximal abelian in $\Cst_{r}(G;\E)$.
  \item\label{it:prop:max:S_April23}  
  $S$ satisfies Conditions~\eqref{eq:thm:May17:Int Omega=1} and ~\eqref{eq:thm:May17:icc} of
  Theorem~\ref{thm:May17}%
  , that is,
  \repeatable{both-conditions}{%
  \begin{enumerate}[wide = 6pt, widest = {\bfseries9999}, leftmargin =*]
      \item[\eqref{eq:thm:May17:Int Omega=1}]$\go\subset
      \Int[\E]{\Omega_{S}\inv(\set{1})}=\E_{S}$ and
      \item[\eqref{eq:thm:May17:icc}]
        $\Int[\E]{\Omega_{S}\inv (\Z_{>1})}=\emptyset$.
    \end{enumerate}%
    }%
  \item\label{it:prop:max:S in terms of Omega + maximal} $S$ is maximal among open subgroupoids of $\Iso{G}$ for which $\E_{S}$ is abelian,  and $S$ satisfies Condition~\eqref{eq:thm:May17:icc}. %
  \end{enumerate} 
\end{prop}

\begin{remark}
    If $\Int[G]{\Iso{G}}$ is abelian, then it was shown in \cite{CRST:2021:Reconstruction}*{Corollary 5.4} (see also \cite{CDRT:2024:CRSTcorrigendum-pp}) that $i(\cs_r(\Int[G]{\Iso{G}}))$ is maximal abelian in $\cs_{r}(G)$. In light of this, it can be conjectured that Proposition~\ref{prop:max_new} also holds if $S$ is merely assumed to be closed in $\Int[G]{\Iso{G}}$ rather than in $G$. 
    However, we assume $S$ to be closed in $G$ because our proof makes use of the set-equality
    in Proposition~\ref{prop:support in S}.
\end{remark}

For the proof of Proposition~\ref{prop:max_new}, we will need
some preliminary results.
Firstly, we have the following on the groupoid level,
which should be compared to \cite{DGNRW:Cartan}*{Lemma 3.1} and \cite{BG:2023:Gamma-Cartan-pp}*{Proposition 3.4}.

  \begin{lemma}\label{lem:S max implies}
     Assume that $S$ is
      maximal among open subgroupoids of $\Int[G]{\Iso{G}}$ for which
      $\E_{S}$ is abelian. 
      If $X\subset \Iso{G}$ is open and for each $e\in \pi\inv(X)$ and $\sigma\in \E_{S}(r(e))$, we have $e\sigma=\sigma e$, then $X\subset S$.
  \end{lemma}

  \begin{proof}
    We adapt the proof of \cite{BG:2023:Gamma-Cartan-pp}*{Proposition 3.4} to the case of general twists. Fix $e\in \pi\inv(X)$. Since $G$ is \etale\ and $X$ is open, there exists an open bisection $Y\subset X$ around $\pi(e)$. Let $\mathcal{F}\coloneq \langle \E_{S}, \pi\inv(Y)\rangle$ be the subgroupoid of $\Iso{\E}$ generated by $\E_{S}$ and $\pi\inv(Y)$. By assumption on $X$ and since $Y\subset X$, every element of $\pi\inv(Y)$ commutes with every element of $\E_{S}$. Moreover, if $\tau_{1},\tau_{2}\in \pi\inv(Y)$ are composable, then since $Y$ is a bisection contained in $\Iso{G}$, there exists $z\in\T $ such that $\tau_{1}=z\cdot \tau_{2}$; since the $\T$-action is central, this means that $\tau_{1}$ and $\tau_{2}$ commute. Since $\E_{S}$ is abelian by assumption, it follows that $\mathcal{F}$ is an abelian subgroupoid of $\E$. 

    Since $(\pi(e_{1}),\pi(e_{2}))\in G\comp $ implies that
    $(e_{1},e_{2})\in \E\comp $, it follows from \cite{AMP:IsoThmsGpds}*{Proposition
    10} that 
    $\pi(\mathcal{F})$ is a
    subgroupoid of $\Iso{G}$, and it follows from Lemma~\ref{lem:int is sbgpd}\ref{item:saturation still
      abelian} that $\pi\inv (\pi(\mathcal{F}))$  is abelian.    
    By Lemma~\ref{lem:int is sbgpd}\ref{it:int(H)}, $T\coloneq \Int[G]{\pi(\mathcal{F})}$ is an open subgroupoid of $G$ and
    \[
        \Int[\E]{\pi\inv(\pi(\mathcal{F}))} 
        \overset{\text{L.~\ref{lem:int is sbgpd}}}{=} \E_{T}
        \subset
        \pi\inv (\pi(\mathcal{F})),
    \]
    which implies that the twist $\E_{T}$ is abelian.

    By definition of $\mathcal{F}$, it contains $\E_{S}$ and $\pi\inv(Y)$. Since $S$ and $Y$ are open, we conclude $\E_{S},\pi\inv(Y)\subset \Int[\E]{\mathcal{F}}$, so $S$ and $Y$ are subsets of $\pi(\Int[\E]{\mathcal{F}})$. Since $\pi$ is open, this set is contained in $\Int[G]{\pi(\mathcal{F})}=T$, so $S,Y\subset T$.  Since $T$ is an open subgroupoid of $\Iso{G}$ and $\E_{T}$ is abelian, maximality of~$S$
    implies that $S=T$. Thus $\pi(e)\in Y\subset T=S$. Since $\pi(e)\in X$ was arbitrary, this proves that $X\subset S$.
  \end{proof}

\begin{corollary}[of Lemma~\ref{lem:S max implies}]\label{cor:S max in terms of Omega}
    The following are equivalent:
    \begin{enumerate}[label=\textup{(\roman*)}]
      \item\label{item:cor:S max implied by} 
      $S$ satisfies Condition~\eqref{eq:thm:May17:Int Omega=1}, i.e., 
      $G\z\subset S$ and
      $\E_{S}=\Int[\E]{\Omega_{S}\inv(\set{1})}$.
      \item\label{item:cor:S max implies} $S$ is
      maximal among open subgroupoids of $\Int[G]{\Iso{G}}$ for which
      $\E_{S}$ is abelian.
\end{enumerate}
\end{corollary}

\begin{proof}
    Note first that, if $T$ is a subgroupoid of $G$ for which $\E_{T}$ is abelian, then $\E_{T}$ is contained in $\Iso{\E}$ and any $\tau\in\E_{T}$ satisfies
    \[
        \operatorname{Ad}_{T}(\tau)
        =
        \set{
            \sigma\inv \tau  \sigma : \sigma\in \E_{T}
        }
        =
        \set{\tau}.
    \]
    If $T$ is open, then $\E_{T}$ is contained in $\Int[\E]{\Iso{\E}}$, so that $\E_{T}\subset \Int[\E]{\Omega_{T}\inv(\set{1})}$.

    \ref{item:cor:S max implied by}$\implies$\ref{item:cor:S max implies}: Let $T$ be
    open and
    as above. If $S\subset T$, then $T\z=\go$ and
    $\E_{S}\subset \E_{T}$, so that
    \begin{align*}
        \Omega_{T}\inv (\set{1})
        &=
        \bigsqcup_{u\in \go}
        \set{
            e\in \E(u):\forall\tau\in \E_{T}(u), e\tau=\tau e
        }
        \\        
        &\subset
        \bigsqcup_{u\in \go}
        \set{
            e\in \E(u):\forall\sigma\in \E_{S}(u), e\sigma=\sigma e
        }
        =\Omega_{S}\inv (\set{1}).
    \end{align*}
    The previous paragraph
    shows
    $(\dagger)$ in the following:
    \[
        \E_{S}\subset \E_{T} 
        \overset{(\dagger)}{\subset} \Int[\E]{\Omega_{T}\inv\set{1}}
        \subset \Int[\E]{\Omega_{S}\inv(\set{1})}
        \overset{\ref{item:cor:S max implied by}}{=}
        \E_{S},
    \]
    so  $\E_{S}=\E_{T}$ and thus $S=T$.

    \ref{item:cor:S max implies}$\implies$\ref{item:cor:S max implied by}:    
    Since $G$ is \etale, $\go$ is clopen, and so $T\coloneq S\cup \go$
  is a clopen subgroupoid of $\Int[G]{\Iso{G}}$ for which the
  restricted twist $\E_{T}$ is abelian. By maximality of~$S$, we conclude $S\z=\go$. 
    If 
    $e_0\in \Int[\E]{\Omega_{S}\inv(\set{1})}$, then 
    we can find an open neighborhood $V$ of $e_0$ such that $V\subset \Omega_{S}\inv(\set{1})$. In particular, $V\subset \Iso{\E}$ and for each $e\in V$ and composable $\sigma\in \E_{S}$, we have  $\sigma\inv e\sigma =e$. Now consider $X\coloneq \pi(V)$. If $e'\in \pi\inv(X)$, then there exists $e\in V$ and $z\in\T $ such that $e'
    =z\cdot e$. In particular,
    \[
    \sigma\inv e'\sigma =
    \sigma\inv (z\cdot e)\sigma
    =
    z\cdot (\sigma\inv e\sigma)
    =
    z\cdot e = e'.
    \]
    This shows that every element $e'$ of $\pi\inv(X)$ likewise commutes with all composable elements of $\E_{S}$.
    By Lemma~\ref{lem:S max implies},
    our assumption in \ref{item:cor:S max implies} implies that
    $X\subset S$, so $\pi\inv(X)\subset \E_{S}$.
    Since $e\in V\subset  \pi\inv(X)$ was an arbitrary point of $\Int[\E]{\Omega_{S}\inv(\set{1})}$, this
    proves that $\Int[\E]{\Omega_{S}\inv(\set{1})}\subset \E_{S}$. Since the reverse containment holds by the first paragraph, that concludes the proof.
\end{proof}

\begin{remark}\label{rmk:cocycle symmetric}
    Applying Corollary~\ref{cor:S max in terms of Omega}\ref{item:cor:S max implies}
    to $2$-cocycles, the only change is that the phrase ``\emph{\ldots for which $\E_{S}$ is abelian}'' becomes ``\emph{\ldots
    which are abelian and
    for which $c$ is symmetric}''. 
\end{remark}

Next, we study
what it means for $a\in \Cst_r (G;\E)$ to
  commute with $i(\Cst_r (S;\E_{S}))$. For example, if $S\z=\go$, then
  by \cite{ren:irms08}*{Theorem 4.2(a)}, the open
  support of $j_{G}(a)$ must be contained in $\Int[\E]{\Iso{\E}}$.
  As a first step, we need 
  a strengthening of \cite{DGNRW:Cartan}*{Lemma
  3.7}.
  We take advantage of
  Remark~\ref{rem-con-defined}
  to prove the following, very helpful lemma
which should  be compared to \cite[Lemma 6.3.3]{Chakraborty:2024:thesis}.

\begin{lemma}\label{lem:conjugation equality}
  Suppose 
  $S\subset \Iso{G}$ and let $h \in C_0(G;\E)$ 
  with $\suppo(h)\subset \Iso{\E}$.
  The following are equivalent.
  \begin{enumerate}[label=\textup{(\roman*)}]
      \item\label{it:conj equ:commutes} For all $f \in C_c(S;\E_{S})$, 
      we have $h*i(f) = i(f)* h$.
      \item\label{it:conj equ:constant on conjugates}  For all
        $u \in S\z $, $e \in \E(u) $, and $\sigma \in \E_{S}(u)$, we have $h(e) = h(\sigma\inv e\sigma)$.
  \end{enumerate}

\end{lemma}
\begin{proof}
    Fix $u \in S\z $, $e \in \E(u) $, and $\sigma \in \E_{S}(u)$.
  If
  $f \in C_c(S;\E_{S})$, then
  \begin{equation}\label{eq:h*i(f)}
      [h * i(f)](e\sigma) = \sum_{\set{\pi(e_{1})\in G:r(e_{1})= u }}
    h(e_{1})i(f)( e_{1}\inv e\sigma).
  \end{equation}
    If $f\in C_{c}(G;\E)$ is such that $f(\sigma)\neq 0$ and  $\pi(\suppo(f))$ is a bisection of
  $G$, then at most one summand in~\eqref{eq:h*i(f)} survives, namely
  that for $\pi(e_{1})$ with $\pi(e_{1}\inv e\sigma)=\pi(\sigma)$ so
  that
  \[ [h *i(f)](e\sigma) = h(e) f(\sigma).
  \]
  Similarly, in the sum
  \[ [i(f) *h](e\sigma) = \sum_{\set{\pi(e_{1}) \in G: r(e_{1})=r(e)}}
    i(f) (e_{1})\ h( e_{1}\inv e\sigma),
  \]
  only the summand for $\pi(e_{1})=\pi(\sigma)$ survives, which yields
  \[ [i (f) * h](e\sigma) = f(\sigma) h(\sigma\inv e\sigma).
  \]
    
  \ref{it:conj equ:commutes}$\implies$\ref{it:conj equ:constant on conjugates}:
  Fix $u,e,\sigma$ as above.
  Since $S$ is open and $G$ is \etale, $S$ is likewise \etale, so by Lemma~\ref{lem:Urysohn}, there exists an element $f$ 
  of $C_{c}(S;\E_{S})$ such that $f(\sigma)\neq 0$ and $\pi(\suppo(f))$ is a bisection.
  With the previous computation, we conclude that
    \[ h(e) f(\sigma)= [h*i(f)](e\sigma)\overset{\ref{it:conj equ:commutes}}{=} [i(f)*h](e\sigma)= f(\sigma) h(\sigma\inv e\sigma).\] 
    Since  $f(\sigma)\neq 0$, we deduce $h(e) = h(\sigma\inv e\sigma)$.

    \ref{it:conj equ:constant on conjugates}$\implies$\ref{it:conj equ:commutes}: 
    For any $f\in C_c(S;\E_{S})$ with $\pi(\suppo(f))$ a bisection,
    first note that $h*i(f)$ and $i(f)*h$ both vanish on $\E\setminus\Iso{\E}$, since we assume $h$ and $f$ to be supported in the isotropy, and they also both vanish on $\E(u)$ for $u\notin s(\suppo(f))$. So suppose $u\in s(\suppo(f))$; let $\sigma\in \suppo(f)\cap \E_{S}(u)$. For $e\in \E(u)$,
    we have by the previous computation that
    \[
    [h*i(f)](e\sigma)=h(e) f(\sigma)\overset{\ref{it:conj equ:constant on conjugates}}{=}  f(\sigma) h(\sigma\inv e\sigma) = [i(f)*h](e\sigma)
    .\]
    Since any element $e'$ of $\E(u)$ can be written as $e'=(e'\sigma\inv)\sigma$, the above proves that the $h*i(f)$ and $i(f)*h$ coincide on all of $\E(u)$.
    Since such $f$ span $C_c(S;\E_{S})$ by Lemma~\ref{lem-bisec-span}, we conclude 
    \ref{it:conj equ:commutes}. 
\end{proof}

 As \cite[Theorem 4.2(i)]{ren:irms08} also holds in the case that the groupoid is not second countable, we deduce:
\begin{corollary}
    Suppose 
  $\go\subset S\subset \Iso{G}$. An element $a$ of $\Cst_r (G;\E)$ commutes with $i(\Cst_r (S;\E_{S}))$ if and only if 
  \[
  j_{G}(a)(e)=
  \begin{cases}
  0 &\text{ whenever } e\notin \Iso{\E},
  \\
      j_{G}(a)(\sigma\inv e\sigma) &\text{ whenever } (e,\sigma)\in (\Iso{\E}\times\E_{S})\cap \E\comp.
    \end{cases}
  \]
\end{corollary}

The next lemma
is the key step in the proof of
Proposition~\ref{prop:max_new}.

\begin{lemma}[technical lemma]\label{lem:fancy lemma}
   Assume that $S\subset \Iso{G}$
   is abelian and
   fails to satisfy Condition~\eqref{eq:thm:May17:icc}, i.e.,
   \[
   \Int[\E]{\Omega_{S}\inv (\Z_{>1})}\neq \emptyset.
   \]
    Then there exist $n\in\Z_{>0}$, pairwise disjoint open bisections $V_{0},\ldots,V_{n} \subset \Iso{G}\setminus S$ and  $\T$-equivariant homeomorphisms \[
        \Psi_{i}\colon \pi\inv (V_i) \to\pi\inv (V_0)
    \] 
   such that: 
       If  $e\in\pi\inv(V_{i})$ for some $i\in\set{0,\ldots, n}$ is composable with $\sigma \in \E_{S}$, then there exists a unique $j\in\set{0,\ldots, n}$ such that $\sigma\inv e\sigma\in\pi\inv(V_{j})$, and we have $\Psi_{i}(e)=\Psi_{j}(\sigma\inv e\sigma)$.
\end{lemma}

\begin{proof}
    We use the notation $\operatorname{Ad}_S$ and $\Omega_{S}$ from Definition~\ref{def:Ad_S,Omega}. 
    We start by assuming that $S\subset\Iso{G}$ is abelian and that $X\subset \Iso{\E}$ is a non-empty open set with $X\cap \E_{S}=\emptyset$ and such that
    \begin{equation}\label{eq:new:property of X}
    N < \left|
    \operatorname{Ad}_{S}(\pi(e))
    \right|
        =
        \left| 
        \operatorname{Ad}_{S}(e)
        \right|
        < \infty
        \quad\text{for all }e\in X
    \end{equation}
    for some $N\in\{0,1,\ldots\}$.
    For any $z\in \T$ and $e\in\Iso{\E}$, we have
    \begin{equation}
        \label{eq:Ad of ze}
    \operatorname{Ad}_{S}(z\cdot e) = \set{ z\cdot e' : e'\in \operatorname{Ad}_{S}(e)},
    \end{equation}
    and thus in particular  $\Omega_{S}(z\cdot e)=\Omega_{S}(e)$. We can thus replace $X$ with $\pi\inv(\pi(X))$ without changing its property 
    at~\eqref{eq:new:property of X}.
    Note
    that~\eqref{eq:new:property of X}
    implies that $U\coloneq r(X)=s(X)$ is contained in~$S\z$, since $\Omega_{S}(e)=0$ for $e\in \Iso{\E}$ with $s(e)\notin S\z$. 
    Since $X\cap \E_{S}=\emptyset$ by assumption,
    we have that $V_0\coloneq\pi(X)$ is contained in $  G\setminus S$.

    By shrinking $X$, we can assume that $V_0$ is a (necessarily open) bisection in $G$. 
  Since $V_0\subset G\setminus S$ and $V_0=\pi(X)\subset \pi(\Int[\E]{\Iso{\E}})$, it follows from 
  Lemma~\ref{lem:int is sbgpd}\ref{item:int of IsoE vs int of IsoG}
  that $V_0\subset \Int[G]{\Iso{G}}\setminus S$.
   Let $g_0(\cdot)\colon U\to V_0$ be the  homeomorphism that sends $u\in U\subset S\z$
   to the unique element $g_0(u)$ in $V_0\cap G(u)\subset \pi(X)$.
    
    \begin{claim}\label{claim:lsc}
    The map $u\mapsto |\operatorname{Ad}_{S}(g_0(u))|$ is lower semi-continuous from $U$ to $\Z_{\geq 0}$.
    \end{claim}

    \begin{proof}[Proof of Claim~\ref{claim:lsc}]
    Fix any $\overline{u}\in U
    \subset S\z
    $ and let $g_0\coloneq g_0(\overline{u})$. 
    It will suffice to find an open neighborhood $\tilde{U}\subset U$ of $\overline u$ such that
    \begin{equation}\label{eq:lsc inequality}
    | \operatorname{Ad}_{S}(g_0(u))|
    \geq 
    |
   \operatorname{Ad}_{S}(g_0(\overline{u}))|
    \end{equation} 
for all $u\in \tilde{U}$.
If $|\operatorname{Ad}_{S}(g_0)|=N+1$, then since $g_0(u)\in \pi(X)$, Inequality~\eqref{eq:lsc inequality} follows automatically from the property at~\eqref{eq:new:property of X} of $X$, and we are done; therefore, we may assume that $n\coloneq |\operatorname{Ad}_{S}(g_0)|-1 
> N\geq 0
$,
so that 
    \[
        \operatorname{Ad}_{S}(g_0)
        =
        \set{ g_0,g_1,\ldots,g_{n}}
        \subset
        G(\overline{u})
    \]
    for pairwise distinct elements $g_{i}$.
     Since $G$ is Hausdorff,
    there is an open neighborhood $\tilde{U}$ of $\overline u$
    such that there exist 
    open
    bisections 
    $V_{1},\ldots, V_{n}$ for which $g_{i}\in V_{i}$ and $\tilde{U}=s(V_{i})$  for all $i\in\set{1,\ldots, n}$. We let $g_i(\cdot)$ denote the homeomorphisms from $\tilde{U}$ to $V_i$ sending $u$ to the unique element $g_i(u)\in V_i\cap G_u$; in particular, $g_{i}(\overline{u})=g_{i}\in G(\overline{u})$.

\begin{figure}[h]
    \centering
    \includegraphics[height=0.6\textheight]{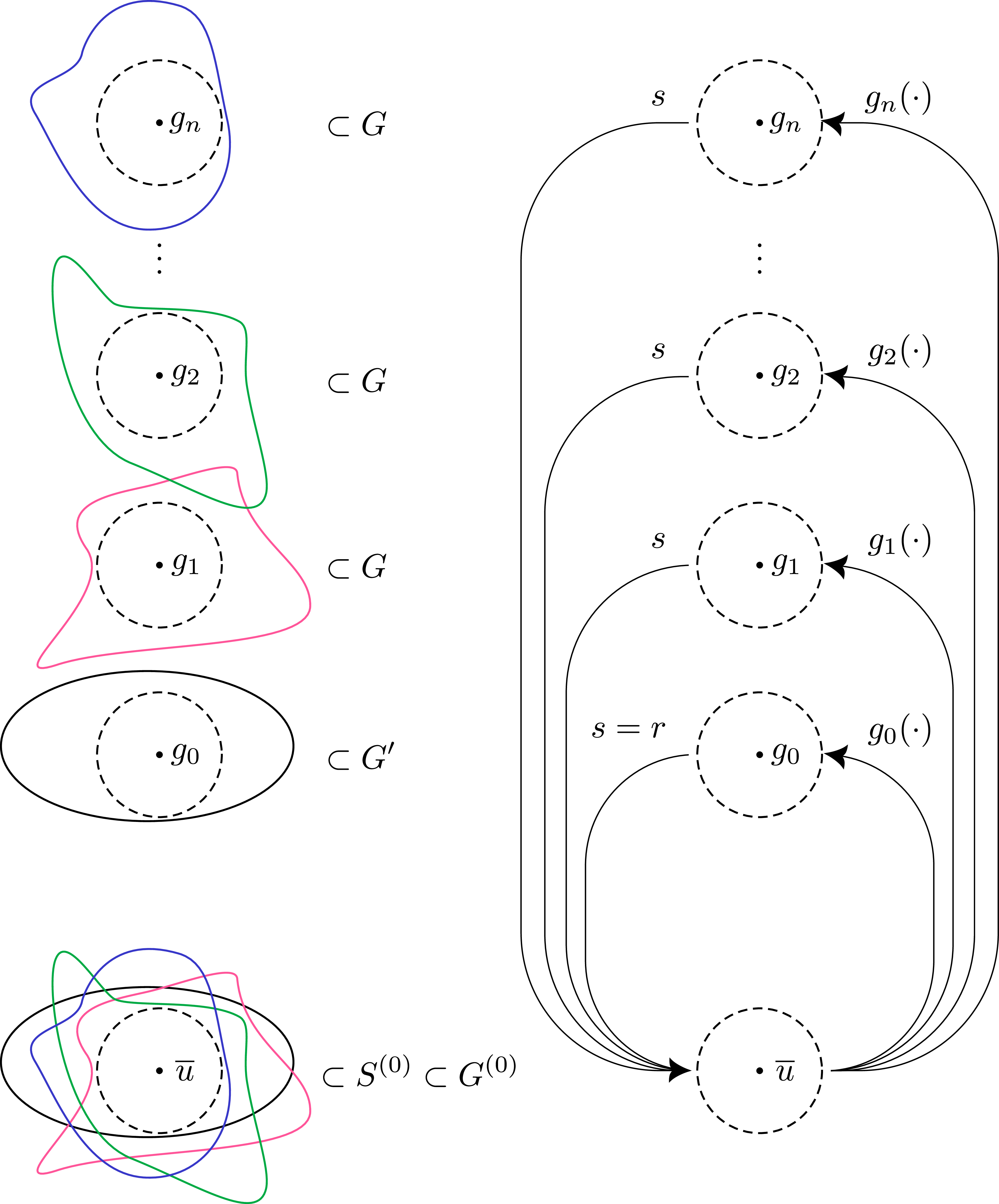}
    \caption{We shrink the open bisections (%
    drawn with solid lines
    in color on the left) around the $S$-conjugates $g_{0},\ldots, g_{n}$ to make them disjoint and homeomorphic (%
    circular neighborhoods drawn with dashed lines
    ).}
    \label{fig:step1}
\end{figure}

    For any two distinct $i,j\in\set{0,\ldots,n}$, we have $g_{i}(\overline{u})=g_{i}\neq g_{j}=g_{j}(\overline{u})$, so by continuity 
    we can shrink $\tilde{U}$
    and
    assume that $g_{i}(u)\neq g_{j}(u)$ for all $u\in \tilde{U}$, meaning that the bisections $V_0\cap r\inv (\tilde{U}),V_{1},\ldots, V_{n}$ are pairwise disjoint; see Figure~\ref{fig:step1}.
    For each $u\in \tilde{U}$,
    we have $s(g_0(u))=u=r(g_0(u))$, so for each $i\in\set{0,\ldots, n}$, we may 
    consider the following pairwise disjoint subsets of $S(u)$: 
    \begin{align*}
    S_{i}(u)
    =
    \set{
    t\in S(u): t\inv g_0(u)t=g_{i}(u)
    },
    \end{align*}
    so that
    \begin{align}        
    \set{ t\inv g_{0}(u) t : t\in S_{i}(u)}
    =
    \begin{cases}
    \emptyset,
    &
    \text{if }S_{i}(u)=\emptyset,
    \\
    \set{g_{i}(u)}
    &
    \text{otherwise.}
    \end{cases}
    \end{align}
    We claim that we can shrink $\tilde{U}$ 
    in such a way that each $S_{i}(u)$ is non-empty
    for $u\in \tilde{U}$.

For each $i\in\set{1,\ldots,n}$, we know by assumption on $g_0=g_0(\overline{u})$ that  $S_{i}(\overline{u})$ is non-empty, so pick $s_{i}\in S_{i}(\overline{u})$.
By choosing small enough open bisections $W_{i}\subset S
\subset \Iso{G}
$ around $s_{i}$ and by shrinking 
the neighborhood $\tilde{U}$ of $\overline{u}$,
we get homeomorphisms $s_i (\cdot) \colon 
\tilde{U}
\to W_i$ sending $u$ to the unique element $s_i(u)\in W_i\cap G(u)\subset S(u)$;
in particular, $s_{i}(\overline{u})=s_{i}$.
Since $u\in S_0(u)$ for all $u\in \tilde{U}\subset S$, we may further let $W_0\coloneq \tilde{U}$ and let $s_0(\cdot)\colon \tilde{U}\to W_0$ be defined by $s_0(u)=u$.

Since $g_{i}(u)\in G_u$ and $s_{i}(u)\in S(u)$ for each $u\in \tilde{U}$, we may
    consider the continuous map
    \[
        \kappa_i\colon \tilde{U}\to 
        G
        \quad\text{ given by }\quad
        u \mapsto
        s_{i}(u)\inv g_0(u)s_{i}(u) g_{i}(u)\inv.
    \]
    By construction of $s_{i}(\cdot)$, we have that $s_{i}(\overline{u})=s_{i}$ is an element of $S_{i}(\overline{u})$, meaning that
    \[\kappa_{i}(\overline{u})=s_{i}\inv g_0s_{i}g_{i}\inv=\overline{u}\in\go,
    \]
    so by \etale ness, $\kappa_{i}\inv(\go)$ is an open neighborhood of $\overline{u}$. By 
    replacing $\tilde{U}$ with its open subset $\kappa_{i}\inv(\go)$, 
    we can without loss of generality assume that $\kappa_{i}(\tilde{U})\subset \go$ for all $i\in\set{0,\ldots, n}$. In other words,
    \begin{equation*}     
    s_{i}(u)\inv g_{0}(u) s_{i}(u)
    =
    g_{i}(u)
    \quad\text{for all } u\in \tilde{U},
    \end{equation*}
    as depicted in Figure~\ref{fig:step2}. In particular,
    $S_{i}(u)$ is non-empty, as claimed.

\begin{figure}[h]
    \centering
    \includegraphics[height=0.6\textheight]{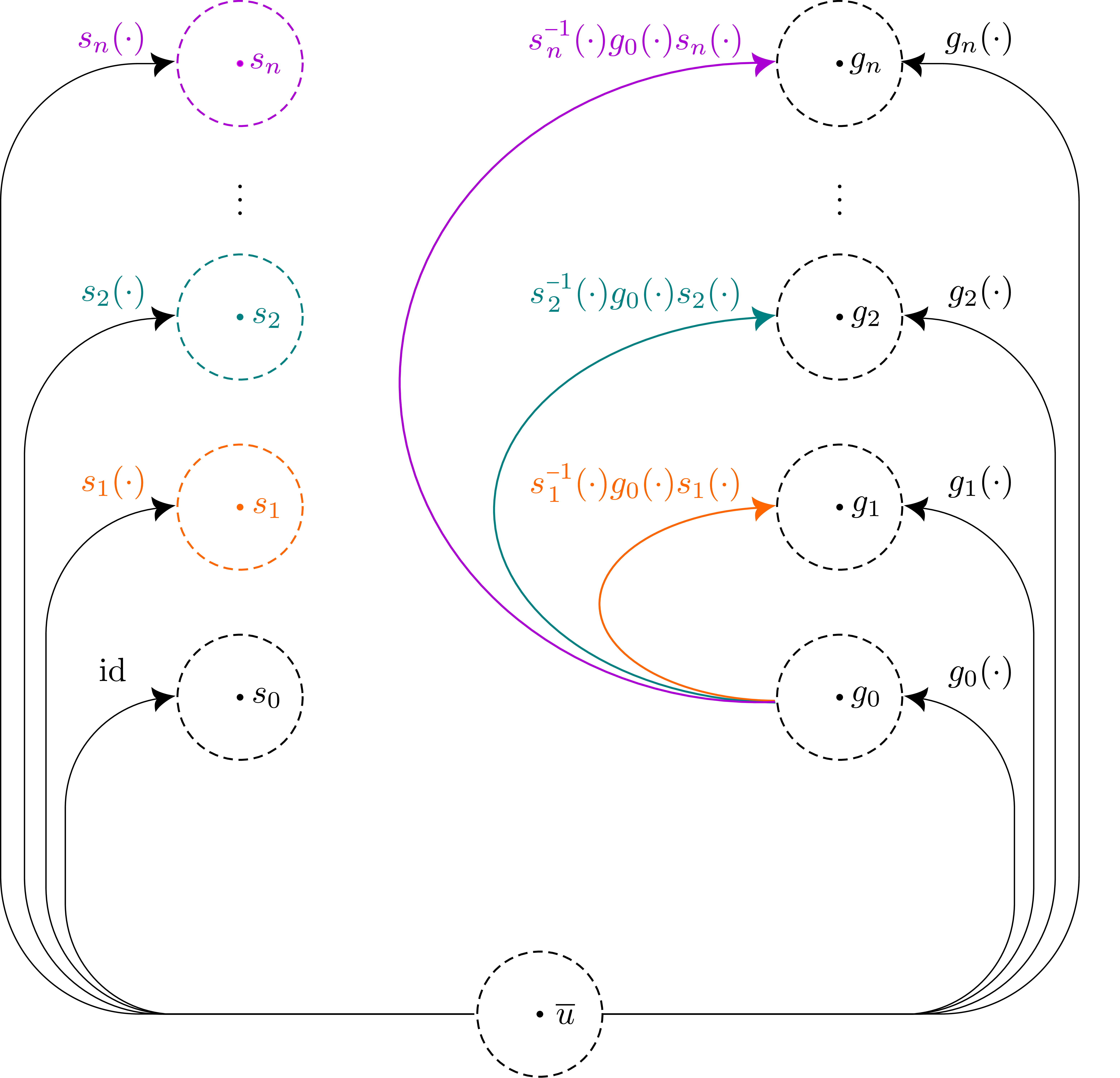}
    \caption{Conjugating $g_0(u)$ by $s_{i}(u)$ yields $g_i(u)$.}
    \label{fig:step2}
\end{figure}
    
    We conclude for each $i\in\set{0,\ldots,n}$ that
    \begin{align}\label{eq:sqcup V_i, finer}
    \Iso{G}
    \supset
    \set{
            t\inv g_{0}(u) t : u\in \tilde{U}, t\in S_{i}(u)
        }
        \supset    
    \set{ g_{i}(u) : u\in \tilde{U}}
    =
    V_{i}\cap r\inv(\tilde{U}).
    \end{align}
 Since $g_{i}(u)\neq g_{j}(u)$ for all $u\in \tilde{U}$ and $i\neq j$ in $\set{0,\ldots,n}$
 and since $S$ is abelian,
 it follows from $s_{i}(u)\inv g_{0}(u) s_{i}(u)
    =
    g_{i}(u)$
     that $g_{i}(u)\in \Iso{G}\setminus S$, so $
     V_{i}\cap r\inv(\tilde{U})
     \subset \Int[G]{\Iso{G}}\setminus S$,  and that
   \[
   |\operatorname{Ad}_{S}(g_0(u))|\geq n+1.
   \]
We have shown: for any fixed element $\overline{u}$ of $U$, we have a neighborhood $\tilde{U}\subset U$  such that 
\[
    | \operatorname{Ad}_{S}(g_0(u))|
    \geq 
    |
   \operatorname{Ad}_{S}(g_0(\overline{u}))|
\]
for all $u\in \tilde{U}$.
Since $\overline{u}$ was arbitrary, this finishes our proof.
    \hfill\qedhere\text{ (Claim~\ref{claim:lsc})}
\end{proof}

\begin{claim}\label{claim:= n+1}
    There exists
    $n\geq N$
    and a non-empty open set $\tilde{X}\subset X$ such that
    \begin{equation}\label{eq:property of tilde-X,new}
    \left|
    \operatorname{Ad}_{S}(\pi(e))
    \right|
        =
        \left| 
        \operatorname{Ad}_{S}(e)
        \right|
        = n+1
        \quad\text{for all }e\in \tilde{X}.
    \end{equation}
\end{claim}
 \begin{proof}[Proof of Claim~\ref{claim:= n+1}]
    As any $e\in X$ already satisfies the equality of set-cardinalities, it suffices to force the equality
    \[
         \left|
    \operatorname{Ad}_{S}(\pi(e))
    \right|
        = n+1.
    \]
    By lower semi-continuity (Claim~\ref{claim:lsc}), for each 
    $n\geq N$,
    the set
    \[
        C_{n} = \{ u\in U: |\operatorname{Ad}_{S}(g_0(u))|\leq n+1\}
    \]
    is closed in $U$. By our assumption on $X$ 
    at~\eqref{eq:new:property of X},
    we have
    \[
     U = \bigcup_{%
        n\geq N    
    }
C_{n}.
    \]
    Since $\go$ is \LCH, and since $U$ is nonempty and open in $\go$, it is a Baire space \cite{Kelley_book}*{34 Theorem}.  Therefore at least one of the closed sets has an interior point \cite{Schechter_book}*{20.15}. If we let $n$ be the smallest element 
    in $\{N,N+1,\ldots\}$
    for which $\Int[\go]{C_n}=\Int[U]{C_n}$ is non-empty, then there exists
   $\overline{u}\in \Int[\go]{C_n}$ with $|\operatorname{Ad}_{S}(g_0(\overline{u}))|=n+1$. By lower semi-continuity (Claim~\ref{claim:lsc}), there exists an open neighborhood $\tilde{U}\subset U$ of $\overline{u}$ such that, for all $u\in \tilde{U}$, we have
    \[
        |\operatorname{Ad}_{S}(g_0(u))|
        \geq 
        |\operatorname{Ad}_{S}(g_0(\overline{u}))|
        = n+1
        .
    \]
    By definition of $C_n$, this implies that, for all $u$ in the non-empty, open set $\tilde{U}\cap \Int[\go]{C_n}$, we have
    \[
        |\operatorname{Ad}_{S}(g_0(u))|
        =
        n+1
        .
    \]
    We may now let $\tilde{X}\coloneq
        r_{\E}\inv \left( \tilde{U}\cap \Int[\go]{C_n}\right)
    \subset r_{\E}\inv(U) = X$. 
    \hfill\qedhere\text{ (Claim~\ref{claim:= n+1})}
    \end{proof}

Assume now that Condition~\eqref{eq:thm:May17:icc} fails, i.e., there exists a non-empty set $X\subset  \Iso{\E}$ that is open in $\E$ and 
    such that Equation~\eqref{eq:new:property of X} holds for $N=1$, i.e.,     
    \begin{equation}\label{eq:property of X}
    1 < \left|
    \operatorname{Ad}_{S}(\pi(e))
    \right|
        =
        \left| 
        \operatorname{Ad}_{S}(e)
        \right|
        < \infty
        \quad\text{for all }e\in X.
    \end{equation}
    Note that~\eqref{eq:property of X} implies that $X\cap \E_{S}=\emptyset$, since $|\operatorname{Ad}_{S}(s)|=1$ for any $s\in S$ because $S$ is abelian. Thus, we may invoke the above claims; in particular, by
    replacing $X$ with its non-empty open subset $\tilde{X}$ from Claim~\ref{claim:= n+1} and by scavenging the proof of Claim~\ref{claim:lsc} (i.e., shrinking $X$ even further), we can now make the following assumptions:
\begin{itemize}
    \item The number $n\in \Z_{>0}$ is such that every element $e$ in the open set $\tilde{X}=X$ satisfies~\eqref{eq:property of tilde-X,new} (not just~\eqref{eq:property of X});
    \item for each $i\in\set{0,\ldots, n}$, we have pairwise disjoint
    open
    bisections $V_{i}\subset 
    \Int[G]{\Iso{G}}\setminus S
    $
    and $W_{i}\subset S$, and homeomorphisms $g_{i}(\cdot)\colon U\coloneq r(X)\to V_{i}$  and  $s_{i}(\cdot)\colon U\to W_{i}$ 
    such that $\pi(X)=V_0$ and
    \begin{equation}\label{eq:s_i conjugate}        
    s_{i}(u)\inv g_{0}(u) s_{i}(u)
    =
    g_{i}(u)
    \quad\text{for all } u\in U;
    \end{equation}
    \item for $i=0$, we have $W_{0}=U$ and $s_{0}(u)=u$.
\end{itemize}
     By the property at~\eqref{eq:property of tilde-X,new} of $X=\tilde{X}$ and the fact that the $V_{i}$ are mutually disjoint, we conclude 
     from
     Equation~\eqref{eq:s_i conjugate}     that
    \begin{equation}\label{eq:Ad_S g_0(u)}
        \operatorname{Ad}_{S}(g_0(u))
        = \set{ g_0(u),\ldots, g_n(u)}
        \subset
        G(u)
    \end{equation}
    for all  $u\in U$.   
    This in turn implies that sets
    \begin{align}\label{eq:S_i(u), def}
    S_{i}(u)
    =
    \set{
    t\in S(u): t\inv g_0(u)t=g_{i}(u)
    },
    \end{align}
    which are non-empty by Equation~\eqref{eq:s_i conjugate}, are a partition of $S(u)$.
    Indeed, if $s\in S(u)$, then $s\inv g_0(u) s\in  \operatorname{Ad}_{S} (g_0(u))$, so  $s\inv g_0(u) s=g_i(u)$ for some $i\in\set{0,\ldots,n}$, meaning that $s\in S_i (u)$. Thus,
    \begin{align}\label{eq:sqcup S_i}
        S(u) = \bigsqcup_{i=0}^{n} S_{i}(u)
    \quad\text{for all } u\in U.
    \end{align}

    Next, we will deal with the $\E_{S}$-conjugates by `lifting' the maps $g_i(\cdot)$.
   By shrinking~$U$, we can without loss of generality assume that the bisections $V_0$ and $W_0,\ldots,W_n$ allow
   local trivializations
   of the principal $\T$-bundle $\E$, meaning that there exist  $\T$-equivariant homeomorphisms $\psi\colon \T\times V_{0} \to  \pi\inv(V_{0})$ 
     and, for each $i\in\set{
     1
     ,\ldots,n}$, 
     $\phi_i\colon \T\times W_{i}\to \pi\inv(W_{i})\subset \E_{S}$ such that the diagrams
    \begin{equation}\label{diags:psi and phi_i}
    \begin{tikzcd}[column sep = small]
        \T\times V_{0} \ar[rr, "\psi"]\ar[rd,"\mathrm{pr}_{2}"'] && \pi\inv(V_{0})\ar[ld, "\pi"]
        \\
        & V_{0}
    \end{tikzcd}
    \quad\text{and}\quad
    \begin{tikzcd}[column sep = small]
        \T\times W_{i} \ar[rr, "\phi_{i}"]\ar[rd,"\mathrm{pr}_{2}"'] && \pi\inv(W_{i})\ar[ld, "\pi"]
        \\
        & W_{i}
    \end{tikzcd}
    \end{equation}
    commute.
    Since $W_0=U\subset \go$,  we may choose for $i=0$ the trivialization $\phi_0\colon \T\times W_{0}\to \pi\inv(W_0)$ given by $\phi_0 (z,u)=\iota(z,u)$.
    For $i\in\set{0,\ldots, n}$,  the following maps are injective and continuous:
    \begin{align}
    \sigma_{i}(\cdot)&\colon U\to \pi\inv (W_{i})
    &\text{ defined by }&&\sigma_{i}(u)&\coloneq \phi_{i} (1,s_{i}(u)), \text{ and}
    \\
     e_{i}(\cdot)&\colon U \to \pi\inv(V_i)
    &\text{ defined by }&&
        e_{i}(u) &\coloneq \sigma_{i}(u)\inv \psi (1,g_{0}(u))\sigma_{i}(u).
    \end{align} 
    Since each $\sigma_i(\cdot)$ takes values in $\E_{S}$, the $e_{i}(u)$ are, by their very definition, $\E_S$-conjugates of each other. We therefore have 
    the following equality of sets:
        \begin{equation}
            \operatorname{Ad}_{S}(e_i(u))=\operatorname{Ad}_{S}(e_j(u)).
        \end{equation}
     For any $i$, commutativity of the diagrams at~\eqref{diags:psi and phi_i} yields that
     \begin{align*}
         \pi(e_{i}(u))
         &=
         \pi\left(
         \sigma_{i}(u)\inv \psi (1,g_{0}(u))\sigma_{i}(u)
         \right)
         &&\text{by definition of $e_{i}(u)$}
         \\
         &=
         \pi\left(
         \phi_{i} (1,s_{i}(u))
         \right)
         \inv 
         \pi\left(\psi (1,g_{0}(u))\right)
         \pi\left(
         \phi_{i} (1,s_{i}(u))
         \right)
         &&\text{by definition of $\sigma_{i}(u)$}
         \\
         &=
         s_{i}(u)\inv g_0(u) s_{i}(u)
         &&\text{by~\eqref{diags:psi and phi_i}}
         \\
         &=g_{i}(u)
         &&\text{by Equation~\eqref{eq:s_i conjugate}.}
     \end{align*}
     Thus, for all $i\neq j$, it follows from $\pi(e_{i}(u))=g_{i}(u)\neq g_{j}(u)=\pi(e_{j}(u))$ that 
    $e_{i}(u)\neq e_{j}(u)$. Since~\eqref{eq:property of tilde-X,new} implies for each $u\in U$ that
    \[
    \Omega_{S}(e_{i}(u))=\Omega_{S}(e_{0}(u))
    =
    |\operatorname{Ad}_{S}(g_0(u))|
    =
    n+1,
    \]
    we conclude that 
    \begin{align}\label{eq:Ad of e_i = of e_0}
        \operatorname{Ad}_{S}(e_i(u))=\operatorname{Ad}_{S}(e_0(u))
        =
        \set{e_{0}(u),\ldots, e_{n}(u)}.
    \end{align} 

For each $i\in\set{0,\ldots,n}$, we define a $\T$-equivariant homeomorphism
\[
\Psi_{i}\colon \pi\inv(V_{i})\to \pi\inv(V_0)
\quad\text{ by }\quad
e\mapsto \sigma_{i}(r(e)) \, e \, \sigma_{i}(r(e))\inv.
\]
Note that every element $e$ of $\pi\inv(V_{i})$ can be written as $e=z\cdot e_{i}(u)$ for some unique $z\in\T$ and $u=r(e)\in r(V_{i})= U$, since $\pi(e)\in V_{i}\cap G(u)=\set{g_{i}(u)}$.
Since we assumed that $s_0(u)=u$
and $\phi_0 (z,u)=\iota(z,u)$ for all $u$ and $z\in\T$,
we further have that $\sigma_0(u)=
\iota(1,u) \in \E\z
$ for all~$u$, so that we can write $\Psi_{i}$ as follows:
\begin{align}
\Psi_{i}(z\cdot e_{i}(u))
&= 
z\cdot 
\sigma_{i}(u) e_{i}(u) \sigma_{i}(u)\inv
&&\text{by definition of $\Psi_{i}$}
\notag
\\&= 
z\cdot
\psi (1,g_{0}(u))
&&\text{by definition of 
$e_{i}(\cdot)$
}
\notag
\\
&=
z\cdot
\sigma_{0}(u) e_{0}(u)
\sigma_{0}(u)\inv
&&
\text{by definition of
$e_{0}(\cdot)$
}
\\
\label{eq:Psi_i}
&=
z\cdot e_0(u)
&&\text{%
since $\sigma_0(u)=\iota(1,u)$.
}%
\end{align}

Now assume that $\sigma\in \E_{S}(u)$, so it is composable with $e=z\cdot e_i(u)\in\pi\inv(V_{i})$. Then it
follows from Equations~\eqref{eq:Ad of e_i = of e_0} and~\eqref{eq:Ad of ze} that
\[
    \sigma\inv e \sigma 
    \in
    \operatorname{Ad}_{S}(e)
    =
    \set{z\cdot e_{0}(u),\ldots, z\cdot e_{n}(u)}.
\]
On the other hand, by Equation~\eqref{eq:sqcup S_i}, we have $s_{i}(u)\pi(\sigma)\in S_{j}(u)$ for a unique $j$, so that 
\eqref{eq:S_i(u), def} implies
\begin{align*}
    \pi(\sigma\inv e\sigma)
    &=
    \pi(\sigma)\inv g_{i}(u) \pi(\sigma)
    &&\text{since $\pi(e)=\pi(e_i(u))$}
\\
    &=
    \pi(\sigma)\inv (s_{i}(u)\inv g_{0}(u) s_{i}(u)) \pi(\sigma)
    &&
    \text{by~\eqref{eq:s_i conjugate}}
    \\
    &=
    g_{j}(u)
    &&
    \text{since $s_{i}(u) \pi(\sigma)\in S_{j}(u)$ and by~\eqref{eq:s_i conjugate}}
    &&
    \\
    &=
    \pi(e_{j}(u))&&
    \text{by~\eqref{diags:psi and phi_i}}.
\end{align*}
In total, we have that
\[
    \sigma\inv e\sigma \in \set{z\cdot e_{0}(u),\ldots, z\cdot e_{n}(u)} \cap \T\cdot e_{j}(u),
\]
so $\sigma\inv e \sigma = z\cdot e_{j}(u)$. Hence by~\eqref{eq:Psi_i}, $\Psi_{j}(\sigma\inv e \sigma) = z\cdot e_{0}(u) = \Psi_{i}(e)$.
\hfill\qedhere\text{ (Lemma~\ref{lem:fancy lemma})}   
\end{proof}

Let us record the conclusion of Claim~\ref{claim:= n+1}, appearing in the proof of Lemma~\ref{lem:fancy lemma}.
\begin{lemma}\label{lem:phew}
    Assume that $S\subset \Iso{G}$
   is abelian, and that $X$ is a non-empty open subset of $\Iso{\E}$ with $X\cap \E_{S}=\emptyset$. If there exists $N\in\{0,1,\ldots\}$ such that
    \begin{align}\notag
    N < \left|
    \operatorname{Ad}_{S}(\pi(e))
    \right|
        &=
        \left| 
        \operatorname{Ad}_{S}(e)
        \right|
        < \infty
        &\text{for all }e\in X,
    \intertext{
    then there exists
    $n\geq N$
    and a non-empty open set $\tilde{X}\subset X$ such that}   
    n+1
    =
    \left|
    \operatorname{Ad}_{S}(\pi(e))
    \right|
        &=
        \left| 
        \operatorname{Ad}_{S}(e)
        \right|
        &\text{for all }e\in \tilde{X}.
    \end{align}
\end{lemma}

\begin{proof}[Proof of Proposition~\ref{prop:max_new}]

We have already seen in Corollary~\ref{cor:S max in terms of Omega} that \ref{it:prop:max:S_April23} and \ref{it:prop:max:S in terms of Omega + maximal} are equivalent, and in Example~\ref{ex-abelian} that $\E_{S}$ is abelian  if and only if $B\coloneq i(\Cst_{r} (S;\E_{S}))$ is abelian.
In particular, we can assume that $\E_{S}$ is an abelian  subset of $ \Iso{\E}$, and we only have to check that maximality of $B$ is equivalent to 
  the two conditions
  \txtrepeat{both-conditions}%

  $\lnot$\ref{it:prop:max_new:B}$\implies\lnot$\ref{it:prop:max:S_April23}: 
  Suppose $B$ is not maximal abelian in $A\coloneq \Cst_{r} (G;\E)$, meaning there exists $a\in A\setminus B$ such that $ab = ba$ for all $b\in B$. Let $h\coloneq j_{G}(a)$.
  Since $a\notin B$ and since $S$ is closed in $G$, Proposition~\ref{prop:support in S} asserts that 
  $\suppo(h)\not\subset \E_{S}$.
  
To show $\lnot$\ref{it:prop:max:S_April23}, we may assume that $S$ satisfies Condition~\eqref{eq:thm:May17:Int Omega=1}; we claim that Condition~\eqref{eq:thm:May17:icc} fails because of 
(an open subset of) 
the set $X\coloneq \suppo(h)\cap \E\setminus\E_{S}$.
Since $S$ is closed, $X$ is an open subset of $\E$, and by the first paragraph, $X$ is non-empty. 

By Condition~\eqref{eq:thm:May17:Int Omega=1}, we have $\go\subset S$; in particular,
since $a$ commutes with $B$ and $\go\subset S$, $a$ commutes with $C_0(\go)$, so \cite{ren:irms08}*{Theorem 4.2(a)} implies that $\suppo(h)\subset  \Iso{\E}$, so $X\subset \Iso{\E}$. We thus may evaluate $\Omega_{S}$ at elements of $X$.

We claim that $\Omega_{S}$ does not take on the value $\infty$ on $X$.
Fix $e\in X$; we will first prove that $\pi$ is injective on $\operatorname{Ad}_{S}(e)$.
  If $\sigma\in\E_{S}(r(e))$ is such that $\pi(\sigma\inv e \sigma)=\pi(e)$, then there exists $z\in\T$ such that $\sigma\inv e \sigma = z\cdot e$. Using the assumption that $a$ commutes with every element of $B$, Lemma~\ref{lem:conjugation equality} 
  and $\T$-equivariance of $h=j_{G}(a)$ imply that
  \begin{equation}
      h(e) \overset{\text{L.~\ref{lem:conjugation
      equality}}}{=} h(\sigma\inv e \sigma) = h(z\cdot e) = zh(e).
  \end{equation}
   As
  $e\in X$, we have
  $h(e)\neq 0$, so the above
  implies $z=1$, so $\sigma\inv e\sigma=  e$. Thus,
  $\pi$ is injective on $\operatorname{Ad}_{S}(e)$. 
  
  To conclude that $\Omega_{S}(e)$ is finite, it remains to show that $|\operatorname{Ad}_{S}(\pi(e))|<\infty$.
  Because $h \in C_0 (G;\E)$,
  we may find a compact $K\subset  G$ such that
    \begin{equation}\label{eq:h(e')_new}
    \abs{h  (e')}
    < \abs{h  (e)}
    \text{ for all }e'\in \pi\inv(G\setminus K).
    \end{equation}
    We claim that $I\coloneq \operatorname{Ad}_{S}(\pi(e))$ is contained in the compact set $K$. Indeed, if $e'\in \pi\inv(I)$, then
    $e'=z\cdot \sigma e \sigma\inv$ for some $z\in\T$ and some
    $\sigma\in \E_{S}$, and so
    \[
        |h (e')|
        =
        |z\, h (\sigma e \sigma\inv)|
        =
        |h (\sigma e \sigma\inv)|\overset{\text{L.~\ref{lem:conjugation
      equality}}}{=}|h (e)|.
    \]
    Combined with Equation~\eqref{eq:h(e')_new}, this shows
    that
    $\pi\inv(I)\subset  \E\setminus\pi\inv(G\setminus K) =
    \pi\inv(K)$. In other words, $I=\pi(\pi\inv(I))$ is contained in
    the compact set $K$. Since $I$ is also a subset of the discrete
    set $G(r(e))$, $I$ is finite,     
    as claimed.
    Since $\go\subset S$ furthermore implies that $\Omega_{S}(e)>0$ for all $e\in\Iso{\E}$, we have shown all in all that
    \begin{align}
    0 < \left|
    \operatorname{Ad}_{S}(\pi(e))
    \right|
        &=
        \left| 
        \operatorname{Ad}_{S}(e)
        \right|
        < \infty
        &\text{for all }e\in X.
    \intertext{By Lemma~\ref{lem:phew}, there thus exists
    $n\geq 0$
    and a non-empty open set $\tilde{X}\subset X$ such that}   
    n+1
    =
    \left|
    \operatorname{Ad}_{S}(\pi(e))
    \right|
        &=
        \left| 
        \operatorname{Ad}_{S}(e)
        \right|
        &\text{for all }e\in \tilde{X}.
    \end{align}
    Note that $n$ must be larger than $0$: otherwise, $\tilde{X}$ would be contained in $\Int[\E]{\Omega_{S}\inv (\set{1})}$, which by Condition~\eqref{eq:thm:May17:Int Omega=1} coincides with $\E_{S} \subset \E\setminus X$. In other words, we have
    \[
    \tilde{X}
    \subset 
    \Omega_{S}\inv (\Z_{>1}).
    \]
    Since $\tilde{X}$ is open and non-empty, this shows that Condition~\eqref {eq:thm:May17:icc} fails.

\smallskip

$\lnot$\ref{it:prop:max:S_April23}$\implies\lnot$\ref{it:prop:max_new:B}:
Assume first that 
Condition~\eqref{eq:thm:May17:Int Omega=1} fails. By Corollary~\ref{cor:S max in terms of Omega}, this means that there exists an open subgroupoid $T$ of $\Int[G]{\Iso{G}}$ whose twist $\E_{T}$ is abelian and such that $S\subsetneq T$. 
We have by Proposition~\ref{prop:support in S} that
\[
B=i(\Cst_{r}(S;\E_{S}))
\subset
\set{a\in \Cst_{r}(G;\E) : \suppo(j_{G}(a))\subset \E_{S}}.
\]
In particular, $i(C_c(T;\E_{T}))\not\subset B$, so $B$
is a proper subalgebra of the commutative $i(\Cst_{r}(T;\E_{T}))$ and $B$ is hence not maximal abelian.

    Now assume that Condition~\eqref{eq:thm:May17:icc} fails. Let $n\in\Z_{>0}$, $V_0,\ldots,V_n$, and $\Psi_{0},\ldots, \Psi_{n}$ be as in Lemma~\ref{lem:fancy lemma}.
    Fix $e_{0}\in \pi\inv(V_0)$ and some precompact open $V$ around $\pi(e_0)$ such that $V\subset\overline{V}\subset V_{0}$. Pick an $h_{0}\in C_c(G;\E)$ with $\pi(\suppo(h_0))\subset V$ and such that $h_{0}(z\cdot e_0)=z$ for all $z\in \T$. For each $1\leq i\leq n$, define at $e\in \E$
    \[
        h_{i}(e)    
        =
        \begin{cases}
            h_{0} (\Psi_{i}(e)),
            &
            \text{if }e\in \pi\inv(V_{i}),
            \\
            0, & \text{otherwise}.
        \end{cases}
    \]
    By assumption on the support of $h_{0}$, this is continuous and compactly supported, and since $\Psi_i$ and $h_{0}$ are $\T$-equivariant, so is $h_{i}$. In other words, $h_{i}\in C_{c}(G;\E)$.

    Let $h\coloneq \sum_{i=0}^{n} h_{i}\in C_c(G;\E)$. Since the $V_{i}$'s are mutually disjoint, so are the $\pi\inv(V_{i})$'s. In particular $\suppo(h)=\bigsqcup_{i=0}^{n}\suppo(h_{i})$, which contains the point $e_{0}\in \pi\inv(V_{0})\subset  \E\setminus \E_{S}$. 
    By Proposition~\ref{prop:support in S}, this means 
that $h\notin B$.     
    We claim that $h$ commutes with $B$, proving that $B$ is  not maximal abelian, as desired.

    Since $\suppo(h)\subset \bigsqcup_{i=0}^{n}\pi\inv(V_i)$ and since $V_i\subset \Iso{G}$ by construction (Lemma~\ref{lem:fancy lemma}), we have $\suppo(h)\subset \Iso{\E}$. Thus, according 
    to Lemma~\ref{lem:conjugation equality}, it suffices to show that, for any $u\in\go$,
     $e\in\E(u)$, and $\sigma\in\E_{S}(u)$, we have $h(e)=h(\sigma\inv e \sigma)$. If $h(e)\neq 0$, then there exists a unique $i\in\set{0,\ldots,n}$ such that $e\in \pi\inv(V_i)$.
    By construction of the maps $\Psi_{0},\ldots,\Psi_{n}$ in Lemma~\ref{lem:fancy lemma}, there exists a unique $j\in\set{0,\ldots, n}$ such that $\sigma\inv e\sigma\in\pi\inv(V_{j})$, and we have $\Psi_{j}(\sigma\inv e\sigma)=\Psi_{i}(e)$. By definition of $h$, we conclude
    \[
    h(e)=h_0 (\Psi_i(e)) = h_0 (\Psi_j(\sigma\inv e\sigma)) = h(\sigma\inv e\sigma),
    \]
    as needed. Replacing $e$ with $\sigma e\sigma\inv$ in this argument proves that indeed, $h(e)=h(\sigma\inv e \sigma)$ for all $\sigma$.
\end{proof}

The following will help us prove Theorem~\ref{thm:May17}.

\begin{cor}[of Proposition~\ref{prop:max_new}]\label{cor:Breg}
  If $i(\Cst_{r}(S;\E_{S}))$ is regular in $\Cst_{r}(G;\E)$, then the following statements are equivalent.
  \begin{enumerate}[label=\textup{(\roman*)}]
      
    \item\label{it:cor:Breg:B} $i(\Cst_{r}(S;\E_{S}))$ is a Cartan subalgebra of $\Cst_{r}(G;\E)$.
    \item\label{it:cor:Breg:S:Omega} $S$ satisfies Conditions~\eqref{eq:thm:May17:Int Omega=1} and~\eqref{eq:thm:May17:icc}, and $S$ is closed in $G$.\item\label{it:cor:Breg:S:Omega+max} $S$ is maximal among open subgroupoids of $\Int[G]{\Iso{G}}$ for which
      $\E_{S}$ is abelian, satisfies Condition~\eqref{eq:thm:May17:icc}, and $S$ is closed in $G$.
  \end{enumerate}
\end{cor}

\begin{proof}
    Let $B\coloneq i (\Cst_{r}(S;\E_{S}))$ and $A\coloneq \Cst_{r}(G;\E)$.

    \ref{it:cor:Breg:B}$\implies$\ref{it:cor:Breg:S:Omega+max}:  
    Since $B$ is maximal abelian and since $i (\Cst_{r}(S;\E_{S}))\subset i (\Cst_{r}(T;\E_{T}))$ for any intermediate open subgroupoid $S\subset T\subset G$, it follows from Corollary~\ref{cor:S max in terms of Omega}, \ref{item:cor:S max implies}$\implies$\ref{item:cor:S max implied by}, that $S\z=G\z$. Thus, the
    existence of a conditional expectation $A\to B$ implies by \cite{BEFPR:2021:Intermediate}*{Lemma 3.4} that $S$ is closed
   in $G$.  Since $S$ is clopen in $G$ and $B$ is a Cartan subalgebra of $A$,
   it follows from Proposition~\ref{prop:max_new}, \ref{it:prop:max_new:B}$\implies$\ref{it:prop:max:S in terms of Omega + maximal},  that $S$ satisfies the remaining claimed properties.

   \ref{it:cor:Breg:S:Omega+max}$\implies$\ref{it:cor:Breg:B}:
   Since $S$ is closed, we can apply Proposition~\ref{prop:CondExp} to deduce that there exists a conditional expectation $A \to B$. Since $S\z=\go$ by Corollary~\ref{cor:S max in terms of Omega}, 
   the conditional expectation is faithful. Since $S\subset \Iso{G}$ is clopen in $G$, we can apply Proposition~\ref{prop:max_new}, \ref{it:prop:max:S in terms of Omega + maximal}$\implies$\ref{it:prop:max_new:B}, to deduce that  $B$ is
    maximal abelian in $A$. Since $B$ is assumed to be regular, we conclude that it is Cartan.

    \ref{it:cor:Breg:S:Omega}$\iff$\ref{it:cor:Breg:S:Omega+max}: This is the equivalence in Corollary~\ref{cor:S max in terms of Omega}, just with the added assumption that $S$ be closed and satisfy Condition~\eqref{eq:thm:May17:icc}.
\end{proof}

\subsection{The proof of Theorem~\ref{thm:May17}}
We now focus on the case that $S$ is normal.

\begin{lemma}\label{lem:Bisections are normalisers}
  Assume that $S\subset \Iso{G}$ is normal 
  in $G$
  and that
  $ h \in C_c(G; \E)$ has support equal to the preimage under $\pi$ of
  a bisection. Then 
  $ h * f * h^{*} \in i (C_c(S;\E_{S}))$ for all
  $ f \in i (C_c(S;\E_{S}))$.
\end{lemma}

\begin{proof}
The argument is essentially
unchanged from \cite{DGNRW:Cartan}*{Lemma 3.11}, one just has to account for the twist.
  Because $C_c(G;\E)$ is closed under convolution, it suffices to show
  that $ h * f * h^{*} $ is supported on $\E_S$.  For $e\in \E$, we
  have
  \begin{align*}
    h   *    f   *   h^{*} (e)
    &=
      \sum_{\pi(e_{1}) \in G^{s(e)}} \sum_{\pi(e_{2}) \in G_{s(e_1)}}
      h(e e_{1}e_{2}\inv) \, f (e_{2})  
      \,  \overline{h (e_{1})}.
      \label{Bisections are normalisers
      Equation 1}  
  \end{align*}
  Assume $e\in\suppo(h * f * h^{*} )$. Since $f$ is supported on
  $\E_S$, there exist $e_{1}\in \E^{s(e)}$ and
  $e_{2}\in (\E_{S})_{s(e_{1})}$ such that
  $ee_{1}e_{2}\inv, e_{1}\in \suppo (h)$. Since $S\subset \Iso{G}$,
  the unit $r(e_{2})=s(ee_{1}e_{2}\inv)$ coincides with
  $s(e_{2})=s(e_{1})$. Since $ee_{1}e_{2}\inv$ and $e_{1}$ 
  therefore
  have the
  same source, our assumption on the support of $h$ implies that
  $ee_{1}e_{2}\inv=z\cdot e_{1}$ for some $z\in\T $. Thus, $e=e_{1}(z\cdot e_{2})e_{1}\inv$
  is an element of $\E_{S}$
  by normality of~$S$, as claimed.
\end{proof}

We point out that, in the above proof, we used that $S$ is normal in $G$, not just in $\Iso{G}$ or $\Int[G]{\Iso{G}}$.
As a direct consequence of Lemma~\ref{lem:Bisections are normalisers} combined with Lemma~\ref{lem-bisec-span}, we get the following (cf.\ \cite{DGNRW:Cartan}*{Proposition 3.12}).
\begin{corollary}
    \label{cor:regular if normal}
  Assume $S\subset \Iso{G}$.  If $S$ is normal in $G$, then
  \[ 
        \set{
        a\in C_{c}(G;\E):
        \pi(\suppo(a)) \text{ is a bisection}
        } 
        \subset 
        N(i (\Cst_{r}(S;\E_{S}))),
    \]
    so $ i (\Cst_{r}(S;\E_{S}))$ is regular.
\end{corollary}

This allows us to finally prove our main theorem.

\pagebreak[3]
\begin{proof}[Proof of Theorem~\ref{thm:May17}]
    \ref{it:thm:May17:B:subset}$\implies$\ref{it:thm:May17:B:cap}:
    Follows from Lemma~\ref{lem-bisec-span}.

    \ref{it:thm:May17:B:cap}$\implies$\ref{it:thm:May17:S:Omega+max}: We only need to show that $\E_{S}$ is normal, since the remaining statement is contained in Corollary~\ref{cor:Breg}, \ref{it:cor:Breg:B}$\implies$\ref{it:cor:Breg:S:Omega+max}.
    So fix any $u\in S\z$, $e\in \E_{u}$, and $\sigma\in \E_{S}(u)$; we must show that $e\sigma e\inv \in \E_{S}$. 

    Since there exists an element in $\Cst_{r}(G;\E)$ that does not vanish at $e$, and since the supremum norm is bounded by the reduced norm, the stronger regularity assumption of~\ref{it:thm:May17:B:cap} implies that there exists a normalizer $n\in \Cst_{r}(G;\E)$ of $i( \Cst_{r}(S;\E_{S}))$ such that  $j_{G}(n)(e)\neq 0$ and $\pi(\suppo(j_{G}(n)))$ is a bisection.    

    Fix $f \in i( C_c(S;\E_{S}))$ with $\pi(\suppo(f ))$ a bisection and $f (\sigma)\neq 0$. Since the supports of $j_{G}(n)$ and $f$ contain $e$ respectively $\sigma$ and are mapped to bisections under $\pi$, an easy computation and application of Proposition~\ref{prop-bfpr} shows that 
    \[
     j_{G}(nfn^*)(e\sigma e\inv) =
    (j_{G}(n) * f * j_{G}(n)^*) (e \sigma e\inv) = |j_{G}(n)(e)|^2 f (\sigma).
    \]
    This is non-zero by choice of $n$ and $f $. Since $n$ is a normalizer,  $n f  n^*$ is an element of the subalgebra, which by Proposition~\ref{prop:support in S} implies that $\suppo(j_{G}(nfn^*)) \subset  \E_{S}$; in particular, $e \sigma e\inv \in \E_{S}$, as claimed.

    \ref{it:thm:May17:S:Omega+max}$\implies$\ref{it:thm:May17:B:subset}:
    We have seen in Corollary~\ref{cor:regular if normal} that $B$ is regular since $S$ is normal, so we can apply Corollary~\ref{cor:Breg}, \ref{it:cor:Breg:S:Omega+max}$\implies$\ref{it:cor:Breg:B}, to deduce that $B$ is Cartan. Corollary~\ref{cor:regular if normal} furthermore shows that 
    \[
        \set{
        h\in C_{c}(G;\E):
        \pi(\suppo(h)) \text{ is a bisection}
        } \subset N(B),
    \]
    which is the remaining part of \ref{it:thm:May17:B:subset}.

    \ref{it:thm:May17:S:Omega}$\iff$\ref{it:thm:May17:S:Omega+max}: This follows from  the equivalence in Corollary~\ref{cor:S max in terms of Omega}, since we merely added the same extra assumptions on $S$ to both statements. 
\end{proof}

    Not every Cartan subalgebra in a groupoid $\Cst$-algebra comes from a subgroupoid as in Theorem~\ref{thm:May17}, as the following example by Ian Putnam shows.

\begin{example}[non-exhaustive]
    Consider $G= \mathfrak{S}_{3}$, the symmetric group on $3$ elements. Its (untwisted, reduced) group $\Cst$-algebra $\Cst_{r}( \mathfrak{S}_{3})$ must be non-commutative and, as a vector space, be of dimension $6$. In other words, we must have $\Cst_{r}( \mathfrak{S}_{3}) = \C  \oplus M_2(\C) \oplus \C$. If $D_{2}$ denotes the diagonal subalgebra of $M_{2}(\C)$, then $\C \oplus D_{2}\oplus\C$ is a Cartan subalgebra of $\Cst_{r}( \mathfrak{S}_{3})$ (actually, it is a $\Cst$-diagonal). However, by Lagrange's theorem, no subgroup of $ \mathfrak{S}_{3}$ is of order $4$, so this Cartan subalgebra cannot have arisen as $\Cst_{r}(S)$ for $S\leq  \mathfrak{S}_{3}$.
\end{example}

\section{Corollaries and Applications}\label{sec:applications}

\begin{example}[The effective case]
We recover \cite{ABOCLMR:2023:Local_Bisection-pp}*{Theorem 3.1, $(1)\iff(2)$} and, in particular, \cite{ren:irms08}*{Theorem 5.2}: for a \LCH, \etale\ groupoid $G$, effectiveness is equivalent to $i(C_0(\go))$ being a Cartan subalgebra of $\Cst_r(G;\E)$ for any twist $\E$. 

To see this, fix a twist $\E$ and an element $e\in\Int[\E]{\Iso{\E}}$. Since $ \E_{\go}=\pi\inv(\go)=\iota(\go\times\T)$ and since $\iota$ is central, we have
\[
    \operatorname{Ad}_{\E,\go}(e)
    =
    \set{
        \tau\inv e \tau 
        :\tau\in \E_{\go}
    }
    =
    \set{e},
    \text{ and }
    \operatorname{Ad}_{\E,\go}(\pi(e))=\set{\pi(e)},
\]
where we have added the extra subscript-$\E$ to emphasize the dependence of the set on the twist. This means that any element $e\in\Int[\E]{\Iso{\E}}$ satisfies $\Omega_{\E,\go}(e)=1$. Thus,
\begin{equation}\label{eq:effective case}
    \Int[\E]{\Omega_{\E,\go}\inv(\set{1})}
    = \Int[\E]{\Iso{\E}}
    \quad\text{and}\quad
    \Int[\E]{\Omega_{\E,\go}\inv(\Z_{>1})}
    = \emptyset.
\end{equation}
Since $G\z$ is always closed and a normal subgroupoid of $G$ (it is just a space), and since $G\z$ is also open by the \etale\ assumption, we can use
Corollary~\ref{cor:Breg} (in tandem with Corollary~\ref{cor:regular if normal})
to get the following chain of equivalences:
\begin{align*}
    G\text{ effective }
    &\overset{\text{def}}{\iff} 
    G\z=\Int[G]{\Iso{G}}
    \overset{\text{L.~\ref{lem:int is sbgpd}}}{\iff}
    \text{For any }\E, \E_{G\z} = \Int[\E]{\Iso{\E}}
    \\&
    \overset{\eqref{eq:effective case}}{\iff} 
    \text{For any }\E, 
    \E_{G\z}=
    \Int[\E]{\Omega_{\E,\go}\inv(\set{1})}
    \text{ and }    \Int[\E]{\Omega_{\E,\go}\inv(\Z_{>1})}
    = \emptyset
    \\
    &\overset{
    \ref{cor:Breg}
    }{\iff}
    \text{For any }\E, 
    i(C_0(\go))\text{ is a Cartan subalgebra of }\Cst_r(G;\E).
\end{align*}

\end{example}

The following corollary of
Theorem~\ref{thm:May17}
allows us to avoid checking the technical Condition~\eqref{eq:thm:May17:icc} directly.

\begin{corollary}\label{cor:formerly thm:main}
  Let 
  $\E$ be a twist 
  over a \LCH, \etale\ groupoid~$G$.  Assume that $S$ is maximal among open subgroupoids of
  $\Int[G]{\Iso{G}}$ for which $\E_{S}$ is abelian, and that $S$ is
  closed and normal in $G$. If we have
\begin{equation}   
    \label{eq:icc}
        \forall u\in \go, g\in G(u)\cap \Int[G]{\Iso{G}}:
        |\set{t\inv g t: t\in S(u)}|\in \set{1,\infty},
\end{equation} 
  then
  $i(\Cst_{r}(S;\E_{S}))$ is a Cartan subalgebra of $\Cst_{r}(G;\E)$.
\end{corollary}

\begin{proof}
    By Corollary~\ref{cor:S max in terms of Omega}, we have $S\z=\go$, so Condition~\eqref{eq:icc} implies $\Omega_{S}(e)\in\{1,\infty\}$ for all $e\in
    \pi\inv(\Int[G]{\Iso{G}})
    =
    \Int[\E]{\Iso{\E}}$. Thus
    \[
    \Int[\E]{\Omega_{S}\inv (\Z_{>1})}
    \subset
    \set{
        e\in \Int[\E]{\Iso{\E}}:
        \Omega_{S}(e)\notin\{0,1,\infty\}
    }
    =
    \emptyset,
    \]
    meaning that Condition~\eqref{eq:thm:May17:icc} is satisfied, so the claim follows from
    Theorem~\ref{thm:May17}, \ref{it:thm:May17:S:Omega}$\implies$\ref{it:thm:May17:B:subset}.
\end{proof}

\begin{remark}[Immediately centralizing]\label{rmk:icc vs immcentr}
A subgroupoid $S$ of $G$ is called \emph{immediately centralizing} in $\Iso{G}$ if the following holds.
\begin{align}\label{eq:immcentr}
    \begin{split}
        &\forall u\in \go, g\in G(u), k\in\Z:
        \\
        &\qquad\bigl[\forall s \in S(u), \exists 1 \leq j \leq k, gs^j = s^jg\bigr] \implies \bigl[\forall s \in S(u), gs=sg\bigr].
    \end{split}
    \end{align}
In \cite{DGNRW:Cartan}*{Theorem 3.1} and \cite{BG:2023:Gamma-Cartan-pp}*{Proposition 3.4},  conditions on a subgroupoid $S$ of $G$ are listed that imply that, in the $2$-cocycle twisted $\Cst$-algebra $\Cst_{r}(G,c)$, the subalgebra  $i(\Cst_r(S,c|_{S^{(2)}}))$ is a Cartan. Like in our
Theorem~\ref{thm:May17}\ref{it:thm:May17:S:Omega+max},
$S$ was assumed to be clopen, normal, and maximal; the assumption that $S$ be immediately centralizing was then invoked to prove maximality of the subalgebra.
Let us quickly explain that Condition~\eqref{eq:immcentr} implies Condition~\eqref{eq:icc} in the last corollary, which proves that Corollary~\ref{cor:formerly thm:main} implies \cite{DGNRW:Cartan}*{Theorem 3.1} and \cite{BG:2023:Gamma-Cartan-pp}*{Proposition 3.4}.

Take $g\in G(u)\cap \Int[G]{\Iso{G}}$  such that $|\set{t\inv g t: t\in S(u)}|>1$, 
 so there exists $t\in S(u)$ with $gt\neq tg$. By Condition~\eqref{eq:immcentr} applied to any $k\geq 1$, this means there exists $s\in S(u)$ such that for all $1\leq j\leq k$ we have $gs^j\neq s^jg$; in particular, the set $\set{s^{-j}gs^{j}: 1\leq j\leq k}$ has $k$ many distinct elements. As $k$ was arbitrary, this shows that $\{t\inv g t: t\in S(u)\}$ must be infinite, or in other words $|\set{t\inv g t: t\in S(u)}|=\infty$, as claimed.
    
    In Example~\ref{ex:immcentr is bad}, we show that the converse implication is false: There are pairs of groupoids $S\leq G$ that satisfy~\eqref{eq:icc} but not~\eqref{eq:immcentr}. This shows that, even in the case of twists induced from $2$-cocycles, Corollary~\ref{cor:formerly thm:main} is stronger than its predecessors \cite{DGNRW:Cartan}*{Theorem 3.1} and \cite{BG:2023:Gamma-Cartan-pp}*{Proposition 3.4}. 
\end{remark}

\begin{example}\label{ex:immcentr is bad}
    Let $\mathfrak{S}_3$ denote the symmetric group of degree $3$ and let $\mathfrak{A}_3$ denote the alternating group, a normal abelian subgroup of $\mathfrak{S}_3$. In $\prod_{n\in  \N } \mathfrak{S}_{3}$, consider the element $h$ that has the transposition $(12)$ in each coordinate. In $\oplus_{n\in  \N } \mathfrak{A}_{3}$, consider for each $k\in\N$, the element $s_k$ that has $(123)$ in the $k$-coordinate and the identity element everywhere else. Since  $(123)\inv (12)(123)=(13)$, we see that
    \(
    s_k\inv h s_{k}
    \)
    has the entry $(13)$ in the $k$-component and $(12)$ everywhere else; in particular, $h$ does not commute with $s_{k}$. Consider the following subgroups of $\prod_{n\in  \N } \mathfrak{S}_{3}$:
    \[
        S\coloneq \langle s_{k} : k \in \N\rangle 
        \leq
        G\coloneq
        \langle S,h\rangle.
    \]
    A quick computation shows that $S=\bigoplus_{n\in\N} \mathfrak{A}_{3}$; by construction, $S$ is maximal abelian in $G$. Note that $S$ is normal in $G$ since $\mathfrak{A}_{3}$ is normal in $\mathfrak{S}_{3}$.
    We claim that $S\leq G$ satisfies~\eqref{eq:icc}. 

    Assume that $g\in G\setminus S$; it suffices to prove that the set of~$S$-conjugates of $g$ is infinite. Since $S$ is normal and $h^2=e$, we can write $g=h^{i}s$ for some exponent $i\in \{0,1\}$ and some element $s\in S$. Since $g\notin S$, we must have $i=1$. Since all but finitely components of~$S$ are trivial, we see that all but finitely components of $g$ are given by the transposition $(12)$. The same computation as that done for $h$ then shows that the set $\set{s_{k}\inv g s_{k}:k\in\N}$ is infinite. This finishes our proof that~\eqref{eq:icc} is satisfied. Thus, by Corollary~\ref{cor:formerly thm:main}, we conclude that $i(\Cst_r(S))$ is a Cartan subalgebra of $\Cst_{r}(G)$.

    In Remark~\ref{rmk:icc vs immcentr}, we have shown that Assumption~\eqref{eq:immcentr} implies Assumption~\eqref{eq:icc}. The example here shows that the converse is false: since the two non-trivial elements of the alternating group $\mathfrak{A}_3$ are of order $3$, any non-trivial element of~$S$ is of order $3$, so that $gs^{3}=g=s^{3}g$ for any  $s\in S$ and any $g\in G$ (i.e., the hypothesis in~\eqref{eq:immcentr} is satisfied for $k=3$ and every $g\in G$), despite not every element of $G$ commuting with every element of~$S$ (i.e., the conclusion of~\eqref{eq:immcentr} is not satisfied for $k=3$). In particular,  \cite{DGNRW:Cartan}*{Theorem 3.1} and \cite{BG:2023:Gamma-Cartan-pp}*{Proposition 3.4} do not see that $i(\Cst_{r}(S))$ is a Cartan subalgebra of $\Cst_{r}(G)$.
\end{example}

We point out that, in the case of the trivial twist, 
Condition~\eqref{eq:thm:May17:icc} becomes
  \[
    \Int[G]{\bigl\{
    g\in \Iso{G}: 1< |\set{t\inv g t:t\in S}|<\infty
    \bigr\}} = \emptyset.
  \]
We therefore get the following easy corollary of
Theorem~\ref{thm:May17}.

\begin{corollary}\label{cor:main-thm for trivial twist}
  Let $G$ be a \LCH, \etale\ groupoid.
For an open subgroupoid $S$ of $G$, the following are equivalent.
  \begin{enumerate}[label=\textup{(\roman*)}]
  
    \item\label{it:main-thm for trivial twist:B} $i(\Cst_{r}(S))$ is a Cartan subalgebra of $\Cst_{r}(G)$,
    and the set
    \begin{equation} 
       \set{
        a\in \cs_{r}(G):
        \suppo(j_{G}(a)) \text{ is a bisection}
        } 
        \cap  N(i(\Cst_{r}(S)))
    \end{equation}
    generates $\Cst_{r}(G)$.
      \item\label{it:main-thm for trivial twist:S} $S$ is maximal among abelian open subgroupoids of
  $\Int[G]{\Iso{G}}$, $S$ is closed and normal in $G$, and the set
  \[
    \Bigl\{
    g\in \Iso{G}: 1< |\set{t\inv g t:t\in S}|<\infty
    \Bigr\}
  \]
  has empty interior.
    
  \end{enumerate}
\end{corollary}

\begin{remark}\label{rmk:where does open go}
     The implication \ref{it:main-thm for trivial twist:S}$\implies$\ref{it:main-thm for trivial twist:B} of  Corollary~\ref{cor:main-thm for trivial twist} is a stronger result than applying the main theorem of \cite{DGNRW:Cartan}*{Theorem 3.1}  to a trivial $2$-cocycle. In both theorems, $S$ is assumed to be open, but in \cite{DGNRW:Cartan}*{Theorem 3.1}, the maximality assumption is with respect to all subgroupoids of $\Iso{G}$ (not necessarily open ones and not necessarily in the interior of $\Iso{G}$). See also
    the proof of \cite{BG:2023:Gamma-Cartan-pp}*{Proposition 3.4} for a discussion. In particular, Corollary~\ref{cor:main-thm for trivial twist} recovers a result in \cite{BNRSW:Cartan} about untwisted groupoid $\Cst$-algebras that is not a corollary of \cite{DGNRW:Cartan}*{Theorem 3.1}; see Corollary~\ref{cor:BNRWS}.
\end{remark}

We will now consider the subgroupoid $\Int[G]{\Iso{G}}$ of $G$ (see Remark~\ref{rmk:intG subgpd}). First note that it is always normal. Indeed, suppose $g\in G$ and  $s\in\Int[G]{\Iso{G}}$ are arbitrary and that $s(g)=r(s)$; we must show that $gsg\inv \in \Int[G]{\Iso{G}}$. Let $V\subset G$ and $W\subset \Iso{G}$ be open bisections around $g$ and $s$, respectively. Then $V\inv$ is an open bisection around $g\inv $ and thus 
\[
    VWV\inv =\set{k'\ell k\inv : k',k\in V, \ell\in W, s(k')=r(\ell), s(\ell)=s(k)}
\]
is an open bisection around $g s g\inv $ by \cite{Putnam:Notes}*{Lemma 3.3.1}. Since $V$ is a bisection and $W\subset \Iso{G}$, the set $VWV\inv $ is entirely contained in $\Iso{G}$, which proves that $g s g\inv$ is an interior point of $\Iso{G}$. Thus $gsg\inv \in \Int[G]{\Iso{G}}$, as claimed. We can now deduce the following 
equivalence, the backwards implication of 
which appeared first in 
\cite{BNRSW:Cartan}*{Theorem 4.3, Corollary 4.5}.

\begin{corollary}\label{cor:BNRWS}
    Assume that $G$ is a  \LCH, \etale\ groupoid, and suppose that $\Int[G]{\Iso{G}}$ is abelian. 
    Then
    $i(\Cst_{r}(\Int[G]{\Iso{G}}))$ is a Cartan subalgebra if and only if  $\Int[G]{\Iso{G}}$ is closed in~$G$.
\end{corollary}
\begin{proof}
    Since $S\coloneq \Int[G]{\Iso{G}}$ is abelian, any $g\in \Int[G]{\Iso{G}}$ satisfies
   \[
   \operatorname{Ad}_{S}(g)=\set{t\inv g t:t\in \Int[G]{\Iso{G}}} = \set{g}
   ,\]
   so that
   \[
    X\coloneq
    \Bigl\{
    g\in \Iso{G}: 1< |\operatorname{Ad}_{S}(g)|<\infty
    \Bigr\}
    \subset 
    \Iso{G}\setminus \Int[G]{\Iso{G}}
    .
  \]
  In other words, $\Int[G]{X}=\emptyset$ and so Condition~\eqref{eq:thm:May17:icc} is satisfied. 
  Since $\Int[G]{\Iso{G}}$ is normal in $G$, it 
  follows from 
Corollary~\ref{cor:Breg} (in tandem with Corollary~\ref{cor:regular if normal})
  that it is a Cartan subalgebra if and only if $S=\Int[G]{\Iso{G}}$ is closed.
\end{proof}

We will use this opportunity to mention two more applications to
groupoids twisted by $2$-cocycles
that are already covered by
the theorems in \cite{DGNRW:Cartan}; we  want to include them here
for
convenience, and to fix a gap in {\cite{MFO2022}*{p.\ 2099, Example
    1}}.

\begin{example}\label{ex:rotation}
  On $G=\Z^2$, consider the $2$-cocycle
  $\mathbf{c}_{\theta}\colon G^2\to \T$ defined for a fixed $\theta\in(0,1)$ by
    \[
        \mathbf{c}_{\theta} ((m_{1},m_{2}),(n_{1},n_{2}))=e^{2\pi i
          \theta m_{2}n_{1}}. 
    \]
    Since $G$ is  a discrete abelian group,
    every subgroup is clopen and normal, so that
    Corollary~\ref{cor:Breg} (in tandem with Corollary~\ref{cor:regular if normal} and with Remark~\ref{rmk:cocycle symmetric})
    states that
    a subgroup $S$ of $G$  will give rise to a Cartan subalgebra
      $\Cst_{r}(S,\mathbf{c}_{\theta}\restr{S^2})$ of
      $\Cst_{r}(G,\mathbf{c}_{\theta})$ if and only if 
    $S$ 
    is maximal subject to the condition that
      $c\restr{S^2}$ is symmetric.  Let us describe these
      maximal subgroups.  Note first that every $S\leq G$ can be
    written as $\Z\cdot (k,0)\oplus \Z\cdot (m,n)$, where either
    $k\ge0$ and $m=n=0$, or $k=0$ and $n>0$, or $k>0$, $n>0$ and
    $m\ge0$.
    
  If $\theta\notin \Q$, then $\Z\cdot (k,0)\oplus \Z\cdot (m,n)$ with
  $k,n>0$ does not give rise to a Cartan subalgebra, since its
  $2$-cocycle is not symmetric and hence its \cs-algebra not
  abelian.
  If $kn=0$, then the remaining subgroups are of the
    form, $\Z\cdot (m,n)$, and give rise to a regular abelian
    subalgebra of $\Cst_r(G,\mathbf{c}_{\theta})$.  However, only
    $\Z\cdot (1,0)$, $\Z\cdot (0,1)$, and $\Z\cdot (m,n)$ with
    $\gcd(m,n)=1$ are \emph{maximal} with symmetric $2$-cocycle.
    (The
  latter subgroups were forgotten in \cite{MFO2022}*{p.\ 2099, Example 1}.)

    This leaves 
    the 
    case that
    $\theta = p/q \in \Q$. If $p=0$,
    then $G$ itself is 
    maximal. If $p\neq 0$, we may assume $\gcd(p,q)=1$. 

    \begin{claim*}
      $\Z\cdot (k,0)\oplus \Z\cdot (m,n)$ is maximal among
    subgroups on which the $2$-cocycle  
        $c_{p/q}$ 
    is symmetric if and only if
    $nk=q$. 
    \end{claim*}

    \begin{proof}[Proof of Claim.]
      It is easy to check that the $2$-cocycle restricted to
      $\Z\cdot (k,0)\oplus \Z\cdot (m,n)$ is symmetric if and only if
      either $k=0$ or $q\mid nk$. In particular, no subgroup of the
      form $\Z\cdot(m,n)$ 
      with $n >0$
      can give rise to a Cartan
      subalgebra, as it as properly contained in, for example,
      $\Z\cdot(q,0)\oplus \Z\cdot(m,n)$ and hence not
      maximal.
      Similarly, if $k>0$, then $\Z\cdot (k,0)\subset
        \Z\cdot (1,0) \oplus \Z\cdot (0,q)$.
        So let
      us focus on the subgroups $\Z\cdot(k,0)\oplus \Z\cdot(m,n)$ with
      $kn > 0$ and $q\mid nk$.

        Note first that $H=\Z\cdot(k,0)\oplus \Z\cdot(m,n)$ is a subgroup of
        $K=\Z\cdot(k',0)\oplus \Z\cdot(m',n')$ if and only if 
    \begin{align}\label{eq:H subgp iff}
        k'\mid k, 
        \quad 
        n'\mid n,
        \quad
        \text{and}
        \quad
        k'n'\mid (mn' - nm').
    \end{align}
    In particular, if $nk=q$ and $m$ is arbitrary, then $H$ is
    maximal: if $H\leq K$ {\em and} $\mathbf{c}_{p/q}$ is symmetric
    on~$K$, then the first two conditions of~\eqref{eq:H subgp iff}
    imply that $n'=n$ and $k'=k$, and the last condition boils down to
    $k\mid (m-m')$, which is symmetric in $m$ and $m'$, forcing
    $H= K$.

    So assume now that $q\mid nk$ but $q<nk$; we must show that $H$ is not
    maximal.
    
    If $k\mathbin{\nmid } q$, then there exists $k'\mid k$, $k'\neq k$,
    such that $q\mid nk'$, so that $\mathbf{c}_{p/q}$ is symmetric on
    $\Z\cdot(k',0)\oplus \Z\cdot(m,n)$ which, by~\eqref{eq:H subgp iff} and
    since $k\neq k'$, properly contains $H$. We can thus assume that
    $q=kn'$ for some $n'\in\N $; note that $q\mid nk$ then means $n'\mid n$,
    and $q< nk$ means $n'< n$.

    If $\gcd(n/n', k)=l > 1$, then $\Z\cdot(k/l,0)\oplus \Z\cdot(m,n)$ properly
    contains $H$. Since $n'l$ divides $n$ by choice of $l$, we see
    that $q=n'k=(n'l)(k/l)$ divides $n(k/l)$, so that the $2$-cocycle
    is symmetric on this supergroup of $H$, proving that $H$ is not
    maximal. We can thus assume that $\gcd(n/n',k)=1$.
    
    By Bezout's identity, there exist $b,c\in \Z$ with $ (n/n')b +
    kc=1$. Multiplying both sides by $n'm$, we  conclude 
    \[
        n'k \mid  (mn' - n(bm)).
    \]
    Together with $n'\mid n$, this implies by~\eqref{eq:H subgp iff} that
    $H\leq \Z\cdot(k,0)\oplus \Z\cdot(bm,n')$ and that, 
    as $n'\neq n$, this
    containment is proper. As $kn'=q$, the $2$-cocycle is symmetric on
    this supergroup, proving that $H$ is not maximal.
  \end{proof}
  
    To sum up: 
    if $\gcd(p,q)=1$, a subgroup $H=\Z\cdot(k,0)\oplus \Z\cdot(m,n)$ of $\Z^2$ 
    gives rise to a Cartan subalgebra of $\Cst_r(G,\mathbf{c}_{p/q})$ if and only if $q=nk$.
\end{example}

\begin{example}
  The example here is a slight simplification of the one given in
  \cite{DGNRW:Cartan}*{Section~7}.
  Consider the set
  $G=\Z /4\Z \times \Z \times\Z $ with the discrete topology. We make
  it a non-abelian group with $2$-cocycle ${c}$ by defining
  \[
    (x,y,z)\cdot (m,n,k)=(x+m+2zn, y+n, z+k) \quad\text{and}\quad
    {c}((m,n,k),(m',n',k'))=(-1)^{nm'}.
  \]
  We will show that $\Cst_{r}(G,{c})$ does not have a Cartan subalgebra that comes from a subgroupoid.

  The maximal abelian subgroups on which ${c}$ is symmetric (i.e.,
  for which the twist is abelian) are
  \[
    S_{1} \coloneq \Z /4\Z \times 2\Z \times \Z \quad\text{and}\quad
    S_{2} \coloneq 2\Z /4\Z \times \Z \times 2\Z .
  \]
  For trivial reasons, these are clopen and have full unit
  space, and since
\begin{align*}
        (m,n,k)\cdot 
        (x,y,z)\cdot 
        (m,n,k)\inv
    &=
    (x+2(ky-zn), y, z),
\end{align*}
we see that both subgroups are normal and that
\[
  \set{ s\cdot (0,1,1)\cdot s\inv \mid s\in S_{i} } = \set{ (0,1,1),
    (2, 1, 1) } .
\] 
Thus, 
\[
 \Omega_{S_{i}}\bigl((0,1,1)\bigr)=|\operatorname{Ad}_{S_{i}}\bigl((0,1,1)\bigr)|=2,
\]
so
\[
 (0,1,1)\in \Omega_{S_{i}}\inv(\Z_{>1}).
\]
In other words, the two maximal subgroups of~$G$ with abelian twist do
not satisfy Condition~\eqref{eq:thm:May17:icc}. Thus, 
Corollary~\ref{cor:Breg}
yields that $\Cst_{r}(S_{1}) $ and $\Cst_{r}(S_{2}) $ are not Cartan subalgebras. And indeed,
$\delta_{(0,1,1)} + \delta_{(2,1,1)}$ commutes with the
\cs-algebras 
$\Cst_{r}(S_{i})$ while not being contained in them,
proving that $\Cst_{r}(S_{1}) $ and $\Cst_{r}(S_{2}) $ are not maximal abelian.
\end{example}

\section{Open questions}\label{sec:Qs}
During the production of this paper, we stumbled upon the following questions.

\begin{question}[{cf.\ Theorem~\ref{thm:May17}\ref{it:thm:May17:B:cap}}]
    If $B$ is a regular subalgebra of $\Cst_{r}(G;\E)$ and $\mathfrak{U}$ is a topology base for~$G$, under which conditions does
    \[
        \set{
        a\in \cs_r(G;\E):
        \suppo(j_G(a)) \in \mathfrak{U}
        } 
        \cap N(B)
    \]
    generate $A$? 
\end{question}

\begin{question}[{cf.\ Remark~\ref{rmk:G amenable} and \cite{FK:2024:Fourier}*{Theorems 2.2 and 2.3}}]
    For which open subgroupoids $S$ of $G$ is the set
    \[
        \set{
        a\in \cs_r(G;\E)
        :
        \suppo(j_G(a))\subset \E_{S}
        }
    \]
    reduced to just being $i(\cs_r(S;\E_{S}))$?
\end{question}

\begin{question}
    We have seen in Corollary~\ref{cor:S max in terms of Omega} that Condition~\eqref{eq:thm:May17:Int Omega=1} is equivalent to $S$ being maximal among open subgroupoids of $\Iso{G}$ for which $\E_{S}$ is abelian. Is there a similarly intuitive characterization of Condition~\eqref{eq:thm:May17:icc}?
\end{question}

\bibliography{DWZbibfile}

\end{document}